\documentclass[12pt]{amsart}

\usepackage{fullpage}

\usepackage[utf8]{inputenc}
\usepackage{xy}
\xyoption{all}

\usepackage[maxbibnames=100]{biblatex}
\newcommand{\letters}{\renewcommand{\theenumi}{\alph{enumi}}}

\newcommand{\AND}{\operatorname{AND}}
\newcommand{\OR}{\operatorname{OR}}

\newcommand{\F}{{\mathbb F}}
\newcommand{\N}{{\mathbb N}}
\newcommand{\R}{{\mathbb R}}
\newcommand{\Z}{{\mathbb Z}}
\newcommand{\T}{{\mathbb T}}
\newcommand{\PP}{{\mathbb P}}
\newcommand{\B}{{\mathbb B}}

\newcommand{\Tr}{\operatorname{Tr}}
\newcommand{\Sym}{\operatorname{Sym}}

\newcommand{\dd}{\operatorname{d}}

\newtheorem{theorem}{Theorem}[section]
\newtheorem{proposition}[theorem]{Proposition}
\newtheorem{lemma}[theorem]{Lemma}
\newtheorem{corollary}[theorem]{Corollary}

\theoremstyle{definition}
\newtheorem{definition}[theorem]{Definition}

\newtheorem{example}[theorem]{Example}

\newcommand{\kar}{\operatorname{char}}
\newcommand{\conv}{\operatorname{conv}}
\newcommand{\slope}{\operatorname{slope}}

\title{Tropical invariants for permutation group actions}
\author{Harm Derksen}
\address{Northeastern University, Boston, USA}
\email{hderksen@northeastern.edu}

\begin{document}
\begin{abstract}
We consider the action of a permutation group $G$ of order $k$ on the tropical polynomial semiring in $n$ variables. We prove that the sub-semiring of invariant polynomials is finitely generated if and only if $G$ is generated by $2$-cycles. There do  exist finitely many separating invariants of degree at most $\max\{n,{n\choose 2}\}$.
Separating tropical invariants can be used to construct bi-Lipschitz embeddings of the orbit space $\R^n/G$ into Euclidean space.
We also show that the invariant polynomials of degree $\leq n p_1p_2\cdots p_k$ generate the semifield of invariant rational tropical functions, where $p_1,p_2,\dots,p_k$ are the first $k$ prime numbers. 
Most results are also true over arbitrary semirings that are additively idempotent and multiplicatively cancellative.
 
\end{abstract}

\maketitle

\tableofcontents

\section{Introduction}
\subsection{The tropical semiring}
The tropical semiring $\T$, also known as the max--plus algebra, is the
 set 
$\R\cup \{-\infty\}$ with
the two binary operations $\oplus$ and $\odot$ defined by $a\oplus b=\max\{a,b\}$ and $a\odot b=a+b$.
The semiring $\T$ is also isomorphic to the min--plus algebra.
This tropical semiring has been studied in many different areas of mathematics, computer science and physics.
The semiring $\T[x_1,x_2,\dots,x_n]$
of {\em tropical polynomials} consists of
all convex piecewise linear functions $\R^n\to \R$ that are obtained from $\R$ and the coordinate functions $x_1,x_2,\dots,x_n$ using the operations $\oplus$ and $\odot$, together with the element $-\infty$.
 Nonzero {\em tropical rational functions} are also piecewise linear functions from $\R^n$ to $\R$, but are not necessarily convex.
In tropical geometry one considers tropical polynomials and their properties to 
study combinatorical aspects in algebraic geometry (see~\cite{maclagan2015introduction}). Tropical geometry also appears in deep learning because deep neural networks with  ReLU activation functions are tropical rational functions (see~\cite{zhang2018tropical}). 

\subsection{Invariant theory}
Suppose that $\F$ is a field and $G$ is an algebraic group that acts on the polynomial ring $\F[x_1,x_2,\dots,x_n]$ by automorphisms. In Invariant Theory 
one studies the invariant ring $\F[x_1,x_2,\dots,x_n]^G$ that consists of all invariant polynomials. If $G$ is a reductive group then the invariant ring is finitely generated. 
In the case where $\F$ has  characteristic 0 this 
was shown by Hilbert~\cite{hilbert1890}, and in positive characteristic this follows from \cite{nagata1963} and \cite{haboush1975}. Emmy Noether showed that the invariant ring is generated by invariants of degree $\leq |G|$ if $\kar(\F)=0$, and this bound also holds when $\kar(\F)>0$ does not divide $|G|$
(see~\cite{fleischmann2000noether,fogarty2001noether}). A slightly weaker bound when $\kar(\F)$ divides $|G|$ was given in~\cite{symonds2011castelnuovo}. If $G$ acts by permuting the variables $x_1,x_2,\dots,x_n$, then the invariant ring is generated by polynomials of degree at most $\max\{n,{n\choose 2}\}$ (see~\cite{gobel1995computing}).

\subsection{Tropical polynomial invariants}
In this paper we explore Tropical Invariant Theory. We will consider 
a finite group $G$ that acts on the 
 tropical polynomial ring $\T[x_1,x_2,\dots,x_n]$ by permuting the variables. Not much is known in this new area. If $G$ is the full symmetric group $S_n$, then the invariant semiring is generated by the elementary symmetric polynomials~\cite[Corollary 3.12]{carlsson2016}.
 This is analogous to the case of the symmetric group $S_n$ acting on the polynomial  ring $R[x_1,x_2,\dots,x_n]$ over a commutative ring $R$ (with 1) in the non-tropical world.
 Unfortunately, invariant semirings are not always finitely generated in the tropical case. 
Consider the symmetric group $S_n$ acting on the tropical polynomial ring $R_{d,n}=\T[\{x^{(i)}_j\}_{1\leq i\leq d,1\leq j\leq n}]$ where $\sigma\cdot x^{(i)}_j=x^{(i)}_{\sigma(j)}$ for all $i,j$ and $\sigma\in S_n$. If $d=n=2$, then the invariant semiring $R_{2,2}^{S_2}$ is not finitely generated
~\cite[Proposition 5.9]{carlsson2016}. The following theorem is a special case of Corollary~\ref{cor:FinitelyGenerated} and Theorem~\ref{theo:NotFinitelyGenerated}.
 Our first main result is:
 \begin{theorem}\label{theo:1}
 If $G\subseteq S_n$ is a subgroup, then  
     the invariant semiring $\T[x_1,x_2,\dots,x_n]^G$ is finitely generated over $\T$ if and only if $G$ is generated by $2$-cycles.     
\end{theorem}
Note that a subgroup of $S_n$ that is generated by $2$-cycles is a product of symmetric groups. From Theorem~\ref{theo:NotFinitelyGenerated} and Corollary~\ref{cor:FinitelyGenerated} follows that this theorem is not only true over the tropical semifield $\T$, but over all semirings that are idempotent and cancellative (see Definition~\ref{def:idempotent_cancellative}).
\subsection{Tropical rational invariants}
     If $G$ acts on $\T[x_1,x_2,\dots,x_n]$ by permuting variables, then it also acts on the semifield of tropical rational functions $\T(x_1,x_2,\dots,x_n)$. 
 It is natural  to ask  whether the invariant semifield $\T(x_1,x_2,\dots,x_n)^G$ is finitely generated as a semifield over $\T$. One special case is the action of $S_n$ on the semifield $F_{d,n}=\T(\{x^{(i)}_{j}\}_{1\leq i\leq d,1\leq j\leq n})$. It was stated in  \cite[Theorem 6.2]{carlsson2016} that $F_{2,n}^G$ is a finitely generated semifield, but a mistake in the proof was pointed out in \cite[p.~85]{kubo2019}.
 Our second main result is:
 \begin{theorem}\label{theo:2}
     If a finite group $G$ of order $k$ acts on $\T(x_1,x_2,\dots,x_n)$ by permuting
     the variables, then $\T(x_1,x_2,\dots,x_n)^G$ is generated as a semifield by all invariant tropical polynomials of degree $\leq n p_1p_2\cdots p_k$ where $p_1,p_2,\dots,p_k$ are the first $k$ prime numbers.
 \end{theorem}
 The proof of Theorem~\ref{theo:2} is in  Section~\ref{sec:Rational}.

\subsection{Bi-Lipschitz Invariant Theory and separating invariants}
Bi-Lipschitz Invariant Theory is a new direction in Invariant Theory about bi-Lipschitz embeddings of quotient spaces into Euclidean space~\cite{agarwal2020nearly,amir2025stability,balan2022lipschitz,balan2022lipschitz,balan2023g,blum2025estimating,cahill2025group,eriksson2018quantitative,haviv2013euclidean,heimendahl2022semidefinite,mixon2023max,mixon2025injectivity,qaddura2024max,vallentin2023least}. Many machine learning algorithms, such as a randomized approximate nearest neighbor search~\cite{jones2011randomized},
apply to data vectors that that live in a Euclidean space $\R^n$. If a model for the data has a group of symmetries $G$, then it usually is more efficient in time and space to exploit the symmetry and  work with data that lies in the quotient space $\R^n/G$. Given a bi-Lipschitz embedding $\R^n/G\hookrightarrow \R^m$ one can apply the machine learning algorithms in Euclidean space to data lying in a quotient space.

 Suppose that $G$ is a compact Lie group acting on the Euclidean space  $\R^n$ by orthogonal transformations. Let $\R^n/G=\{G\cdot v\mid v\in V\}$ be the orbit space.
The space $\R^n/G$ has a metric given by $$\dd(G\cdot v,G\cdot w)=\min_{g,h\in G}\|g\cdot v-h\cdot w\|=\min_{g\in G}\|g\cdot v-w\|.$$
A fundamental problem in bi-Lipschitz Invariant Theory is finding a bi-Lipschitz embedding $\phi:\R^n/G\hookrightarrow \R^m$. The bi-Lipschitz property means that there exist positive constants $C_1,C_2$ such that
$$
C_1\cdot \dd(G\cdot v,G\cdot w)\leq \|\phi(v)-\phi(w)\|\leq C_2\cdot  \dd(G\cdot v,G\cdot w).
$$
for all $v$ and $w$. If $C_1$ is chosen as large as possible and $C_2$ is chosen as small as possible then the ratio $C_2/C_1$ is called the distortion of the embedding. For applications one would like to find a bi-Lipschitz embedding with the least distortion.
As explained in \cite[\S1.2.1]{blum2025estimating}, $\R^n/G$ is a finite dimensional Alexandrov space of negative curvature, 
and using work of Zolotov~\cite{zolotov2019bi} that builds on~\cite{eriksson2018quantitative} there exists a bi-Lipschitz embedding $\R^n/G\hookrightarrow \R^m$. See also \cite[Appendix B]{blum2025estimating} for a short proof that uses less machinery. For a finite group $G$ a randomized construction in~\cite[Theorem~18]{cahill2025group} gives a bi-Lipschitz embedding of $\R^n/G\hookrightarrow \R^m$ with low distortion.

Suppose that $f_i:\R^n\to \R$ is a $G$-invariant function for $i=1,2,\dots,m$ and let $\phi=(f_1,f_2,\dots,f_m):\R^n\rightarrow \R^m$. We say that $f_1,f_2,\dots,f_m$ are {\em separating invariants} if $G\cdot v=G\cdot w$ if and only if $\phi(v)=\phi(w)$. 
The map $\phi:\R^n\to \R^m$ factors through $\overline{\phi}:\R^n/G\to \R^m$. 
Now $f_1,f_2,\dots,f_m$ are separating exactly when $\overline{\phi}$ is injective. If $\overline{\phi}$ is bi-Lipschitz, then $\overline{\phi}$ must be injective, but the converse is not always true. The {\em max filter} with template $z\in \R^n$ is defined as
$$
f(v)=\max_{g\in G}\langle v,g\cdot z\rangle.
$$
Suppose that $z_1,z_2,\dots,z_m\in \R^n$ and define the max filter $f_i:\R^n\to \R$ by $f_i(v)=\max_{g\in G} \langle v,g\cdot z_i\rangle$. Such a sequence $f_1,f_2,\dots,f_m$ is called a {\em max filter bank}.
For $m\geq 2n$ and $z_1,z_2,\dots,z_m$ chosen randomly, it was shown in \cite[Theorem 18]{cahill2025group} that, with positive probability, $f_1,f_2,\dots,f_m$ are separating and $\overline{\phi}$ is bi-Lipschitz. The following theorem will be proved in Section~\ref{sec:Separating}.

\begin{theorem}\label{theo:3}
    If $G$ acts on $\T[x_1,x_2,\dots,x_n]$ by permuting the variables, then there exists a nonnegative integer $m$ and separating tropical invariants $f_1,f_2,\dots,f_m\in \T[x_1,x_2,\dots,x_n]^G$ of degree $\leq \max\{n,{n\choose 2}\}$    
    such that $\phi=(f_1,f_2,\dots,f_m):\R^n\to \R^m$ factors through a bi-Lipschitz embedding $\R^n/G\hookrightarrow \R^m$, where $m=n+n!/|G|$. 
\end{theorem}
The theorem gives an explicit, non-randomized construction of a bi-Lipschitz embedding, but $m$ may be large. Note that any representation $G$ can be embedded into another representation on which $G$ acts by permuting the coordinates. So the problem of finding bi-Lipschitz embeddings can be reduced to the case where $G$ acts by permuting coordinates.

% The authors of \cite{balan2023g} proved that for a max filter bank $f_1,f_2,\dots,f_m$, $\overline{\phi}$ is bi-Lipschitz if and only if $f_1,f_2,\dots,f_m$ are separating.

 \subsection{Generalizations to other semirings}
Instead of just working over the tropical semiring $\T=(\R\cup\{-\infty\},\oplus,\odot)$ we will also generalize our results to more general semirings such as for example the boolean semiring $\B=(\{{\bf 0},{\bf 1}\},\oplus ,\odot)$,
where $a\oplus b=a\OR b$ and $a\odot b=a\AND b$ for all $a,b\in \{{\bf 0},{\bf 1}\}$.

We will always assume that a semiring $(R,\oplus,\odot)$ is commutative with identity elements ${\bf 0}$ and ${\bf 1}$ for addition and multiplication respectively and that ${\bf 0}\neq {\bf 1}$. Note that ${\bf 0}=-\infty$ and ${\bf 1}=0$ in $\T$. 
We use the notations  
$\bigoplus_{i=1}^n a_i=a_1\oplus a_2\oplus \cdots\oplus a_n$, $\bigodot_{i=1}^n a_i=a_1\odot a_2\odot \cdots\odot a_n$
and $a^{\odot n}=\bigodot_{i=1}^n a$
for all $a,b,a_1,a_2,\dots,a_n\in R$ and nonnegative integers $n$. If there is no risk of ambiguity, we will use the abbreviations $ab$ for $a\odot b$ and $a^n$ for $a^{\odot n}$.

\begin{definition}\label{def:idempotent_cancellative}
A semiring $(R,\oplus ,\odot)$ is called (additively) {\em idempotent} if $a\oplus a=a$ for all $a\in R$. It is called (multiplicatively) {\em cancellative } if for all $a,b,c\in \R$ with $c\neq {\bf 0}$ and $ac=bc$ we have $a=b$. We will call a semiring $R$  {\em convex} if it is idempotent {\em and} cancellative.
\end{definition}
The semirings $\T$ and $\B$ are convex. In many ways, convex semirings behave like the tropical semiring $\T$. For example, convex semirings satisfy the Frobenius equality $(a\oplus b)^n=a^n\oplus b^n$ (or freshmen's dream) for all $a,b\in R$ and positive integers $n$ (see~\cite[proof of Lemma 4.3]{connes2011contemporary} or ~Lemma~\ref{lem:TropicalishProperties}). Many of the results in this paper will be generized to convex semirings. For these generalizations, we have to define the polynomial ring over a convex semirings. The tropical polynomial ring $\T[x_1,x_2,\dots,x_n]$ has already been defined as a set of certain convex functions. But note that this tropical polynomial ring is not just the set of formal polynomial expressions over $\T$, because there are non-trivial relations such as ${\bf 1}\oplus x\oplus x^2={\bf 1}\oplus x^2$ in $\T[x]$. Analagous to the tropical polynomial ring $\T[x_1,x_2,\dots,x_n]$ one
 could define the polynomial ring $R[x_1,x_2,\dots,x_n]$ as a set of functions $R^n\to R$. Although $x$ and $x^2$ represent the same function $\B\to \B$, we do not want the relation $x=x^2$ in $\B[x]$, just like the monomials $x$ and $x^2$ in the polynomial ring $\F_2[x]$ over the field $\F_2$ with 2 elements are not the same.

In Section~\ref{sec:tropical} we will construct a polynomial ring $R[x_1,x_2,\dots,x_n]$ for any convex semiring $(R,\oplus,\odot)$. The polynomial ring $R[x_1,x_2,\dots,x_n]$ itself will also be convex, and it will have the following universal property:\\[10pt]
{\it For every convex semiring $S$, homomorphism $\phi:R\to S$ and elements $b_1,b_2,\dots,b_n\in S$ there exists a unique homomorphism $\widehat{\phi}:R[x_1,x_2,\dots,x_n]\to S$ such that $\widehat{\phi}(a)=a$ for all $a\in R$ and $\widehat{\phi}(x_i)=b_i$ for all $i$.}\\[10pt]
The universal property defines $R[x_1,x_2,\dots,x_n]$ up to isomorphism. The tropical polynomial ring $\T[x_1,x_2,\dots,x_n]$ has already been defined and does have this universal property.

We will  show that Theorem~\ref{theo:1} and Theorem~\ref{theo:2} still hold if we replace $\T$ by any convex semiring $R$. It was shown in \cite{kalivsnik2019} that $R[x_1,x_2,\dots,x_n]^{S_n}$ is generated by the elementary symmetric functions for a large class of semirings $R$ that includes all convex semirings.

\section{Tropical algebra}\label{sec:tropical}

\subsection{The tropical semiring}
We will define the tropical semiring $\T$ as the max--plus algebra
$(\R\cup \{-\infty\},\oplus,\odot)$
where $a\oplus b=\max\{a,b\}$ and $a\odot b=a+b$ for all $a,b\in \R\cup\{-\infty\}$. The identity for the addition $\oplus$ is ${\bf 0}:=-\infty$ and the identity for the multiplication $\odot$ is ${\bf 1}:=0$.
Let ${\mathcal F}_n$ the set of all functions from $f:\R^n\to\R$ of the form
\begin{equation}\label{eq:AXB}
f(x_1,x_2,\dots,x_n)=\max\{\alpha_{i,1}x_1+\alpha_{i,2}x_2+\cdots+\alpha_{i,n}x_n+b_i\mid 1\leq i\leq m\},
\end{equation}
where $m$ is a positive integer,  $\alpha_{i,j}\in \N=\{0,1,2,\dots\}$ and $b_i\in \R$ for all $i$ and $j$.
 We define the tropical polynomial ring in $n$ variables as
$$
\T[x_1,x_2,\dots,x_n]=({\mathcal F}_n\cup \{-\infty\},\oplus,\odot)
$$
where $f\oplus g=\max\{f,g\}$ and $f\odot g=f+g$ for all $f,g\in {\mathcal F}_n\cup \{-\infty\}$. The function $f\in \T[x_1,x_2,\dots,x_n]$ in (\ref{eq:AXB}) has the tropical form
$$
f=\bigoplus_{i=1}^m (b_i) x_{1}^{\odot \alpha_{i,1}}\odot x_{2}^{\odot \alpha_{i,2}}\odot \cdots\odot x_n^{\odot \alpha_{i,n}},
$$

We have defined elements in $\T[x_1,x_2,\dots,x_n]$ as functions. This means that we have relations such as
$$
{\bf 1}\oplus x_1\oplus x_1^2=\max\{0,x_1,2x_1\}=\max\{0,2x_1\}={\bf 1}\oplus x_1^2,
$$
even though the polynomials on the left and right are not the same as formal polynomials. One can also define a polynomial ring over $\T$ whose elements are formal polynomial expression and we denote that ring by $\T\{x_1,x_2,\dots,x_n\}$.

\subsection{Idempotent semirings}
We will assume that semirings are commutative, have distinct identity elements for addition and for multiplication. We assume that a semiring $(R,\oplus ,\odot)$ satisfies the following axioms:
\begin{enumerate}
    \item the binary operation $\oplus$ is associative and commutative with identity element ${\bf 0}$;
    \item the binary operation $\odot$ is associative and commutative with identity element ${\bf 1}$;
    \item distributive law: $a (b\oplus c)=ab\oplus ac$ for all $a,b,c\in R$;
    \item ${\bf 0}\odot a={\bf 0}$ for all $a\in R$;
    \item ${\bf 0}\neq {\bf 1}$.
\end{enumerate}
A map $\phi:R\to S$ between semirings is a homomorphism if $\phi({\bf 0})={\bf 0}$, $\phi({\bf 1})={\bf 1}$ and for all $a,b\in R$ we have $\phi(a\oplus b)=\phi(a)\oplus\phi(b)$ and $\phi(ab)=\phi(a)\phi(b)$.

\begin{definition}
    A semiring $R$ is called (additively) {\em idempotent} if $a\oplus a=a$ for all $a\in R$.
\end{definition}
If ${\bf 1}\oplus {\bf 1}={\bf 1}$ in a semiring, then for all $a\in R$ we have $a=a\odot {\bf 1}=a\odot ({\bf 1}\oplus {\bf 1})=(a\odot {\bf 1})\oplus (a\odot {\bf 1})=a\oplus a$ and $R$ is idempotent.
An idempotent semiring is naturally equipped with a partial ordering:
\begin{definition}
For elements $a,b$ in an idempotent semiring $R$ we define $a\leq b$ if and only if $a\oplus b=b$.
\end{definition}
The properties in the following lemma are straightforward:
\begin{lemma}\label{lem:idempotentproperties}
 If $R$ is an idempotent semiring and $a,b,c\in R$, then we have
 \begin{enumerate}
 \letters
     \item $a\oplus b$ is the least upper bound for $a$ and $b$;
     \item if $a\leq b$ then $a\oplus c\leq b\oplus c$;
     \item if $a\leq b$ then $ac\leq bc$;
     \item if $a\leq {\bf 0}$ then $a={\bf 0}$.
 \end{enumerate}
\end{lemma}

If $a,b\in R$ then we have $ab\leq a^2\oplus ab\oplus b^2=(a\oplus b)^2$. This property generalizes to the following lemma:
\begin{lemma}\label{lem:idempotentmonomialbound}
    Suppose that $R$ is an idempotent semiring, $c_1,c_2,\dots,c_n\in R$ and $k_1,k_2,\dots,k_n$ are nonnegative integers. Then we have
    $$
\bigodot_{i=1}^n c_i^{k_i}\leq \Big(\bigoplus_{i=1}^n c_i\Big)^{\sum_{i=1}^n k_i}.
$$
\end{lemma}
\begin{proof}
    Let $k=\sum_{i=1}^n k_i$. If we expand $(\bigoplus_{i=1}^n c_i)^k$ we get the sum of all monomials in $c_1,c_2,\dots,c_n$ of degree $k$.
    In particular, the monomial $\bigodot_{i=1}^n c_i^{k_i}$ appears, so $\bigodot_{i=1}^n c_i^{k_i}\leq (\bigoplus_{i=1}^n c_i)^k$.
\end{proof}
% The boolean semiring $\B=(\{{\bf 0},{\bf 1}\},+,\cdot)$ is defined by $a+b=\max\{a,b\}=a\OR b$ and $a\cdot b=a\AND b$.
% The boolean semiring is clearly idempotent. Any semiring that contains $\B$ as a subring satisfies ${\bf 1}+{\bf 1}={\bf 1}$ and must be idempotent. Conversely, if $(R,+,\cdot)$ is an idempotent semiring, then $(\{{\bf 0},{\bf 1}\},+,\cdot)$ is a subring isomorphic to $\B$. So idempotent semirings are sometimes referred to as $\B$-algebras.

\subsection{Cancellative semirings}

\begin{definition}
    We say that a semiring $(R,\oplus,\odot)$ is a {\em weak domain} if $a\odot b={\bf 0}$ implies $a={\bf 0}$ or $b={\bf 0}$ for all $a,b\in R$.
    An element $a$ in a semiring $R$ is called (multiplicatively)
    {\em cancellative} if $a\odot b=a\odot c$ implies $b=c$ for all $b,c\in R$. The semiring $R$ is called {\em cancellative} if every nonzero element in $R$ is cancellative. The semiring $R$ is a {\em semifield} if every nonzero element in $R$ is invertible. If a semiring $R$ is idempotent {\em and} cancellative, we call it {\em convex}.
\end{definition}
It is easy to see that every semifield is cancellative and 
 every cancellative semiring is a weak domain. We use the term ``weak domain'' because some authors reserve the name ``domain'' for semirings with stronger properties. For example, in~\cite{bertram2017tropical} a semiring is an integral domain if it is cancellative. In~\cite{joo2018dimension} a semiring is a domain if it has the even stronger property that $a_1b_1\oplus a_2b_2=a_1b_2\oplus a_2b_1$ implies that $a_1=a_2$ or $b_1=b_2$.
 We use the term ``convex'' because of the property in Lemma~\ref{lem:convex}.

 If $R$ is a semifield, and $b\in R$ is nonzero, then the multiplicative inverse of $b$ will be denoted by $b^{-1}$. We also write $a\oslash b:=ab^{-1}$.
 The tropical semiring $\T$ and the boolean semiring $\B$ are idempotent semifields. For $n\geq 1$, the tropical polynomial semiring $\T[x_1,x_2,\dots,x_n]$ is convex but not a semifield.

Suppose that $R$ is a weak domain. We can construct a multiplicative cancellative semiring as follows. We define a relation $\sim$ on $R$ by $a\sim b$ if and only if there exists a nonzero $u\in R$ with $ua=ub$. 
It is easy to verify that $\sim$ is an equivalence relation. For $a\in R$ we denote its equivalence class by $[a]$. Let $R^\circ=\{[a]\mid a\in R\}$. We define addition and multiplication in $R^\circ$ by $[a]\oplus [b]=[a\oplus b]$
and $[a]\odot [b]=[a\odot b]$.
The addition and multiplication are well defined and make $R^\circ$ into a cancellative semiring. The identities for addition and multiplication are $[{\bf 0}]$ and $[{\bf 1}]$ respectively. If $a\in [{\bf 0}]$ then there exists a nonzero $u\in R$ with $u\odot a=u\odot {\bf 0}={\bf 0}$ and $a={\bf 0}$ because $R$ is a weak domain. This shows that $[{\bf 0}]=\{{\bf 0}\}$.
We define the quotient map $\pi:R\to R^\circ$ by $\pi(a)=[a]$. Then $\pi$ is a homomorphism of semirings. If $R$ is idempotent, then so is $R^\circ$. The homomorphism $\pi:R\to R^\circ$ has the following universal property:
\begin{lemma}\label{lem:universal}
    If $R$ is a weak domain, $S$ is a cancellative semiring and $\phi:R\to S$ is a homomorphism of semirings with $\phi^{-1}({\bf 0})=\{{\bf 0}\}$, then there exists a unique homomorphism $\overline{\phi}:R^\circ\to S$ such that 
    the diagram
$$
\xymatrix{
R\ar[r]^-\pi\ar[d]_-{\phi} & R^\circ \ar[ld]^-{\overline{\phi}}\\
S & }
$$
 commutes, i.e.,    
    $\overline{\phi}\circ \pi=\phi$.
\end{lemma}
\begin{proof}
    For $[a]\in R^\circ$ (and $a\in R$) we must define $\overline{\phi}([a])=\overline{\phi}(\pi(a))=\phi(a)$. The map $\overline{\phi}$ is well defined:
    If $a\sim b$ and there exists a nonzero element $u\in R$ with $ua=ub$. So we get $\phi(u)\phi(a)=\phi(ua)=\phi(ub)=\phi(u)\phi(b)$. Now $\phi(u)\neq {\bf 0}$ and $\phi(a)=\phi(b)$ because $S$ is an cancellative semiring.
   So there is a unique function $\overline{\phi}:R^\circ\to S$ with $\overline{\phi}\circ \pi=\phi$. It is easy to verify that $\overline{\phi}$ is a homomorphism of semirings.
    \end{proof}

Suppose that $(R,\oplus,\odot)$ is a cancellative semiring. We can construct its quotient semifield $Q(R)$ as follows. Let $S$ be the set of all formal expressions $a\oslash b$ with $a,b\in R$ and $b\neq {\bf 0}$. We define an relation $\equiv$ on $S$ by 
$a\oslash b\equiv c\oslash d$ if and only if $ad=bc$.
It is easy to verify that $\equiv$ is an equivalence relation. Let $[a\oslash b]$ be the equivalence class of an element $a\oslash b\in S$, and let $Q(R)$ be the set of all such equivalence classes. 
We define addition and multiplication in $Q(R)$ by: $[a\oslash b]\oplus [c\oslash d]=[(ad\oplus bc)\oslash (bd)]$ and $[a\oslash b]\cdot [c\oslash d]=[(ac)\oslash (bd)]$. The addition and multiplication are well-defined and make $Q(R)$ into a semifield. The identity elements of addition and multiplication are $[{\bf 0}\oslash {\bf 1}]$ and $[{\bf 1}\oslash {\bf 1}]$ respectively. We define $\iota:R\to Q(R)$ by $\iota(a)=[a\oslash {\bf 1}]$. Then $\iota$ is an injective homomorphism of semirings. So we may view $R$ as a sub-semiring of $Q(R)$. If $R$ is idempotent, then so is $Q(R)$.

\begin{lemma}
    Suppose that $R$ is a cancellative semiring, $L$ is a semifield and $\phi:R\to L$ is a homomorphism of semirings with $\phi^{-1}({\bf 0})=\{{\bf 0}\}$. Then there exists a unique homomorphism $\widehat{\phi}:Q(R)\to L$ with $\widehat{\phi}\circ \iota=\phi$.
\end{lemma}
\begin{proof}
    Define $\widehat{\phi}$ by $\widehat{\phi}([a\oslash b])=\phi(a)\oslash \phi(b)$
    for all $a,b\in R$ with $b\neq {\bf 0}$.
    We show that $\widehat{\phi}$ is well-defined. Suppose that $[a\oslash b]=[c\oslash d]$. Then we have $ad=bc$, and
    $\phi(a)\phi(d)=\phi(b)\phi(c)$ and $\phi(c)$ and $\phi(d)$ are nonzero. It follows that $\phi(a)\oslash \phi(b)=\phi(c)\oslash \phi(d)$. This shows that $\widehat{\phi}$ is well defined. It is easy to verify that $\widehat{\phi}$ is a homomorphism and it is clear that $\widehat{\phi}$ is unique.
\end{proof}

\subsection{Convex semirings}
Suppose that $R$ is a convex semiring.
\begin{lemma}\label{lem:TropicalishProperties}
For $a,b,c\in R$, $n\geq 1$ and $c\neq 0$ then we have
\begin{enumerate}\letters
\item $a\leq b$ $\Leftrightarrow$ $ac\leq bc$;
\item $(a\oplus b)^n=a^n\oplus a^{n-1}b\oplus a^{n-2}b^2\oplus \cdots\oplus b^n$;
        \item $a^n\leq b^n$ $\Leftrightarrow$ $a\leq b$;
        \item $a^n=b^n$ $\Leftrightarrow$ $a=b$;
        \item $(a\oplus b)^n=a^n\oplus b^n$.
    \end{enumerate}
\end{lemma}
\begin{proof}\ 

(a) Lemma~\ref{lem:idempotentproperties}~(c) shows one direction. If $ac\leq bc$ then $(a\oplus b)c=ac\oplus bc=bc$. By the cancellation property, $a\oplus b=b$ and $a\leq b$.

(b) This follows by induction.

(c) If $a\leq b$ then we have $a^n\leq b^n$ for all $n\geq 1$ by induction. Suppose that $a^n\leq b^n$ and $n\geq 1$. Then we have
$$
a(a\oplus b)^{n-1}=a^n\oplus a^{n-1}b\oplus a^{n-2}b^2\oplus \cdots\oplus ab^{n-1}\leq a^{n-1}b\oplus a^{n-2}b^2\oplus\cdots\oplus ab^{n-1}\oplus b^n=b(a\oplus b)^{n-1}.
$$
If $a\oplus b\neq 0$ then we have $a\leq b$ by repeatedly using part (a) where $c=a\oplus b$. If $a\oplus b={\bf 0}$ then we have $a=b={\bf 0}$ and therefore $a\leq b$.

(d) This follows from part (c).

(e) We have 
$$
(a^n\oplus b^n)(a\oplus b)^n=
(a^n\oplus b^n)(a^n\oplus a^{n-1}b\oplus \cdots\oplus b^n)=a^{2n}\oplus a^{2n-1}b\oplus \cdots\oplus b^{2n}=(a\oplus b)^{2n}.
$$
If $a\oplus b\neq {\bf 0}$ then we get $a^n\oplus b^n=(a\oplus b)^n$ by the cancellation property. If $a\oplus b={\bf 0}$ then $a=b={\bf 0}$ and $a^n\oplus b^n={\bf 0}=(a\oplus b)^n$.
\end{proof}

Suppose $c=(c_1,c_2,\dots,c_n)\in R^n$ and $\alpha=(\alpha_1,\alpha_2,\dots,\alpha_n)\in \N^n$. Then we write $c^{\alpha}$ for $c_1^{\alpha_1}c_2^{\alpha_2}\cdots c_n^{\alpha_n}$.

\begin{lemma}\label{lem:convex}
   Suppose $R$ is a convex semiring, $S\subseteq \N^n$ is a finite subset, $\beta\in \N^n$ and $c\in R^n$. If $\beta$ lies in the convex hull of $S$, then we have
   $$
   c^{\beta}\leq \bigoplus_{\alpha\in S}c^{\alpha}.
   $$
 \end{lemma}
\begin{proof}
    There exist nonnegative integers $k_\alpha$, $\alpha\in S$, not all $0$ such that
    $$
    \Big(\sum_{\alpha\in S}k_\alpha\Big)\beta=\sum_{\alpha\in S} k_\alpha \alpha.
    $$
Using Lemma~\ref{lem:idempotentmonomialbound} we obtain
$$
(c^{\beta})^{\sum_{\alpha\in S} k_\alpha}=\prod_{\alpha\in S} (c^{\alpha})^{k_\alpha}\leq \Big(\bigoplus_{\alpha\in S}c^\alpha\Big)^{\sum_{\alpha\in S} k_\alpha},
$$
so it follows that
$c^\beta\leq \bigoplus_{\alpha\in S} c^\alpha$ by Lemma~\ref{lem:TropicalishProperties}(c).
\end{proof}

\subsection{Polynomial semirings and power series semirings}
Suppose $R$ is a semiring. We will first construct the (formal) polynomial semiring and the (formal) polynomial power series semirings over $R$. The construction is analogous to the the construction of the polynomial ring and polynomial power series ring over a ring $R$. We will write $R\{x\}$ and $R\{\!\{x\}\!\}$ for the polynomial semiring and the polynomial power series semiring. The notations $R[x]$ and $R[[x]]$ will be reserved for a slightly different construction.
As a set, $R\{\!\{x\}\!\}$ consists of all 
formal expressions
$$
a(x)=\bigoplus_{k=0}^\infty a_k x^k
$$
where $a_0,a_1,a_2,\dots\in R$. 
If $b(x)=\bigoplus_{k=0}^\infty b_kx^k$
then we define $a(x)\oplus b(x)=\bigoplus _{k=0}^\infty (a_k \oplus b_k)x^k$ and $a(x)\odot b(x)=\bigoplus_{k=0}^\infty c_kx^k$ where
$c_k=\bigoplus_{i=0}^k a_ib_{k-i}$.
Let $R\{x\}\subseteq R\{\!\{x\}\!\}$ be the set of all $a(x)=\bigoplus_{k=0}^\infty a_kx^k$ with
$a_k=0$ for $k\gg 0$.
One can check that  $R\{\!\{x\}\!\}$ 
and $R\{x\}$ are again semirings. 
If $R$ is idempotent, then so are $R\{x\}$ and $R\{\!\{x\}\!\}$. However, if $R$ is cancellative then $R\{x\}$ and $R\{\!\{x\}\!\}$ do not need to be cancellative. 
\begin{example}
In $\B\{x\}\subseteq \B\{\!\{x\}\!\}$ we have
$$
1\oplus x^2\neq 1\oplus x\oplus x^2,\quad\mbox{but}\quad
(1\oplus x)\odot (1\oplus x^2)=1\oplus x\oplus x^2\oplus x^3=(1\oplus x)\odot (1\oplus x\oplus x^2).
$$
This shows that $\B\{x\}$ and $\B\{\!\{x\}\!\}$ are not cancellative, even though $\B$ is.
\end{example}
Inductively, we define a polynomial semiring in $n$ variables by 
$$R\{x_1,x_2,\dots,x_n\}=R\{x_1,x_2,\dots,x_{n-1}\}\{x_n\}.$$
The polynomial semiring has the following universal property:\\[10pt]
{\it If $\phi:R\to S$ is a homomorphism of semirings, and $b_1,b_2,\dots,b_k\in S$, then there exists a unique homomorphism $\widetilde{\phi}:R\{x_1,x_2,\dots,x_n\}$ with $\widetilde{\phi}(a)=a$ for all $a\in R$ and $\widetilde{\phi}(x_i)=b_i$ for all $i$.}
\begin{theorem}\label{theo:polyringiso}
    The semirings $\T[x_1,x_2,\dots,x_n]$ and $\T\{x_1,x_2,\dots,x_n\}^\circ$ are isomorphic.
\end{theorem}
\begin{proof}
    By the universal property of $\T\{x_1,x_2,\dots,x_n\}$ there exists a unique homomorphism $\phi:\T\{x_1,x_2,\dots,x_n\}\to \T[x_1,x_2,\dots,x_n]$
    with $\phi(a)=a$ for all $a\in \T$ and $\phi(x_i)=x_i$ for all $i$. By Lemma~\ref{lem:universal} there exist a unique homomorphism $\overline{\phi}:\T\{x_1,x_2,\dots,x_n\}^\circ\to \T[x_1,x_2,\dots,x_n]$ such that $\overline{\phi}\circ \pi=\phi$, i.e., we have the following commutative diagram
    $$
    \xymatrix{
    \T\{x_1,\dots,x_n\}\ar[r]^\pi\ar[rd]_\phi & T\{x_1,\dots,x_n\}^\circ\ar[d]^{\overline{\phi}}\\
    & T[x_1,\dots,x_n]}
    $$

    It is clear that $\phi,\overline{\phi},\pi$ are all surjective. We will show that $\overline{\phi}$ is also injective. Suppose that $\overline{\phi}(\pi(f))=\phi(f)=\phi(g)=\overline{\phi}(\pi(g))$ for some $f,g\in \T\{x_1,x_2,\dots,x_n\}$. We can write
    $$
    f=\bigoplus_{\alpha \in S} c_\alpha  x^\alpha,\quad
    g=\bigoplus_{\beta\in T}d_\beta  x^\beta$$
    with $S,T\subseteq \N^n$ and 
     $c_\alpha$ and $d_\beta$ are not equal to ${\bf 0}=-\infty$ for all $\alpha$ and $\beta$. 
     As functions from $\R^n$ to $\R$, we have
\begin{multline*}
\max\{\alpha_{1}x_1+\alpha_2x_2+\cdots+\alpha_nx_n+c_\alpha\mid\alpha\in S\}=\phi(f)(x_1,x_2,\dots,x_n)=\\
=\phi(g)(x_1,x_2,\dots,x_n)
=\max\{\beta_{1}x_1+\beta_2 x_2+\cdots+\beta_n x_n+d_\beta\mid \beta\in T\}.
\end{multline*}
So we have
\begin{equation}\label{eq:betaalpha}
\beta_1x_1+\beta_2x_2+\cdots+\beta_nx_n+d_\beta\leq \max\{\alpha_{1}x_1+\alpha_2x_2+\cdots+\alpha_nx_n+c_\alpha\mid \alpha\in S\}
\end{equation}
for all $x_1,x_2,\dots,x_n\in \R$.
Consider the following linear program: \\[10pt]
maximize 
$x_0+\beta_1x_1+\beta_2x_2+\cdots+\beta_n x_n$\\ under the constraints
$x_0+\alpha_1x_1+\alpha_2x_2+\cdots+\alpha_nx_n\leq -c_\alpha$ for all $\alpha \in S$.\\[10pt] 
For the optimal solution $(x_0,x_1,x_2,\dots,x_n)$
there exists an $\alpha\in S$ such that
$$
x_0+\beta_1x_1+\beta_2x_2+\cdots+\beta_nx_n\leq x_0+\alpha_1x_1+\alpha_2x_2+\cdots \alpha_nx_n+c_\alpha-d_\alpha\leq -d_\alpha.
$$
We introduce variables $y_\alpha$, $\alpha\in S$. Then the dual linear program is:\\[10pt]
minimize $\sum_{\alpha\in S}(- c_{\alpha})y_{\alpha}$\\
under the constraints: $\sum_{\alpha\in S} \alpha y_\alpha=\beta$, $\sum_{\alpha}y_\alpha=1$ and $y_\alpha\geq 0$ for all $\alpha$.\\[10pt]
There exists an optimal solution $y_\alpha$, $\alpha\in S$
and we may assume that this solution is rational. We can write $y_\alpha=k_\alpha/k$ where $k$ and $k_\alpha$ are nonnegative integers for all $\alpha$. From $\sum_{\alpha\in S}y_\alpha=1$ and $\sum_{\alpha\in S}\alpha y_\alpha=\beta$
follows that $\sum_{\alpha\in S}k_\alpha=k$ and $\sum_{\alpha\in S}k_\alpha \alpha=k\beta$.
The linear program and its dual have the same optimal value. So we have
  $-\sum_{\alpha}c_\alpha y_\alpha\leq -d_\beta$. It follows that $\sum_{\alpha} k_\alpha c_\alpha\geq k d_\beta$.
  In $\T$ we have that
  $$
  d_\beta^{\odot k}\leq \bigodot_{\alpha\in S} c_\alpha^{\odot k_{\alpha}}.
  $$
 By Lemma~\ref{lem:idempotentmonomialbound} we have 
\begin{multline*}
(d_\beta \odot x^\beta)^{\odot k}=d_\beta^{\odot k} \odot (x^{\beta})^{\odot k}\leq \Big( \bigodot_{\alpha\in S}c_\alpha ^{\odot k_\alpha}\Big)\odot \Big(
\bigodot_{\alpha \in S} (x^\alpha)^{\odot k_\alpha}\Big)=\\=\bigodot_{\alpha\in S} (c_\alpha \odot x^\alpha)^{\odot k_\alpha}\leq \Big(\bigoplus_{\alpha\in S} c_\alpha \odot x^{\alpha}\Big)^{\odot k}=f^{\odot k}
\end{multline*}
in $\T\{x_1,x_2,\dots,x_n\}$. 
Since Lemma~\ref{lem:TropicalishProperties}(e) applies to $R^\circ$, we have
$$
\pi(g)^{\odot k}=
\Big(\bigoplus_\beta \pi(d_\beta\odot x^\beta)\Big)^{\odot k}=
\bigoplus_\beta \pi(d_\beta\odot x^\beta)^{\odot k}=\bigoplus_\beta \pi((d_\beta \odot x^\beta)^{\odot k})\leq \pi(f^{\odot k})=\pi(f)^{\odot k}.
$$
By Lemma~\ref{lem:TropicalishProperties}(d)
we have $\pi(g)\leq \phi(f)$.
Similar reasoning with the roles of $f$ and $g$ interchanged gives $\pi(f)\leq\pi(g)$. We conclude that $\pi(f)=\pi(g)$. 
This proves that $\overline{\phi}$ is injective, so $\overline{\phi}$ is an isomorphism.  
\end{proof}

\subsection{Convex polynomial semirings}
Motivated by Theorem~\ref{theo:polyringiso}, we make the following definition.

\begin{definition}
    Suppose that $(R,\oplus,\odot)$ is an convex semiring. We define the convex polynomial semiring over $R$ by $R[x]:=R\{x\}^\circ$.
\end{definition}
Recall that in the construction of $R\{x\}^\circ$
we introduced an equivalence relation $\sim$ on $\R\{x\}$ by 
 $b(x)\sim c(x)$ if and only there exists a nonzero polynomial $a(x)\in R\{x\}$ with $a(x)b(x)=a(x)c(x)$.
 Then $R[x]$ is the set of all equivalence classes $[b(x)]$ with $b(x)\in R\{x\}$. Suppose $b,c\in R\subseteq R\{x\}$ and $b\sim c$. Then we have
 $a(x)b\sim a(x)c$ for a nonzero polynomial $a(x)$.
 If $a(x)=a_0\oplus a_1x\oplus \cdots \oplus a_kx^k$ then $a_i\neq 0$ for some $i$.
 So we have $a_ib=a_ic$ and $b=c$ by the cancellative property of $R$. By identifying $b\in R$ with $[b]\in R[x]$ we can view $R$ as a sub-semiring of $R[x]$.

\begin{lemma}\label{lem:extendAIMC}
    Suppose $\phi:R\to S$ is a homomorphism between convex semirings.
    \begin{enumerate}
    \letters
        \item We can uniquely extend $\phi$ to a homomorphism $\widehat{\phi}:R[x]\to S$  with $\widehat{\phi}(x)={\bf 0}$.
        \item If $y\neq {\bf 0}$ and $\phi^{-1}({\bf 0})=\{{\bf 0}\}$ then we can uniquely extend $\phi$ to a homomorphism  $\widehat{\phi}:R[x]\to S$ with $\widehat{\phi}(x)=y$. Moreover, we have $\widehat{\phi}^{-1}({\bf 0})={\bf 0}$.
    \end{enumerate}
    \end{lemma}
\begin{proof} \ \\
(a) For a polynomial $a(x)=a_0\oplus a_1x\oplus a_2x^2+\cdots\oplus a_kx^k\in R\{x\}$ we define
$\widehat{\phi}([a(x)])=\phi(a_0)=\phi(a({\bf 0}))$. We have to show that $\widehat{\phi}$ is well-defined. Suppose $b(x)\sim c(x)$ in $R\{x\}$. Then there exists $a(x)\in R\{x\}$ with $a(x)b(x)=a(x)c(x)$. By factoring out a power of $x$ we may assume that $a({\bf 0})\in R$ is nonzero. We get $a({\bf 0})b({\bf 0})=a({\bf 0})c({\bf 0})$. Because $R$ is cancellative, it follows that $b({\bf 0})=c({\bf 0})$
and $\phi(b({\bf 0}))=\phi(c({\bf 0}))$.

(b) The homomorphism $\phi$ uniquely extends to a homomorphism $\widetilde{\phi}:R\{x\}\to S$ with $\widetilde{\phi}(x)=y$.
If $a(x)=a_0\oplus a_1x\oplus a_2x^2\oplus \cdots\oplus a_kx^k$ is nonzero, then $a_i\neq {\bf 0}$ and $\phi(a_i)\neq {\bf 0}$ for some $i$, so $\widetilde{\phi}(a(x))=\phi(a_0)\oplus \phi(a_1)y\oplus\cdots\oplus\phi(a_k)y^k$ is nonzero.
If $b(x),c(x)\in R\{x\}$ and $b(x)\sim c(x)$ then there exists a nonzero $a(x)\in R\{x\}$ with $a(x)b(x)=a(x)c(x)$. It follows that $\widetilde{\phi}(a(x))\widetilde{\phi}(b(x))=\widetilde{\phi}(a(x))\widetilde{\phi}(c(x))$. Because $S$ is cancellative and $\widetilde{\phi}(a(x))$ is nonzero, we get
$\widetilde{\phi}(b(x))=\widetilde{\phi}(c(x))$. 
This proves that $\widetilde{\phi}:R\{x\}\to S$ uniquely factors through some homomorphism $\widehat{\phi}:R[x]\to S$ such that $\widehat{\phi}([a(x)])=
\widetilde{\phi}(a(x))$.  If $a(x)$ is nonzero, then $\widehat{\phi}([a(x)])=\widetilde{\phi}(a(x))$ is nonzero.  
\end{proof}
Inductively we define $R[x_1,x_2,\dots,x_n]:=R[x_1,x_2,\dots,x_{n-1}][x_n]$. The following lemma shows that the symmetric group $S_n$ acts on this convex polynomial ring by automorphisms.
\begin{lemma}\label{lem:permutePoly}
Suppose that $R$ is a convex semiring.
   For any permutation $\sigma\in S_n$ there exists a unique automorphism  $\rho_\sigma$
   of the semiring $R[x_1,x_2,\dots,x_n]$ with $\phi(a)=a$ for all $a\in R$ and $\rho_\sigma(x_i)=x_{\sigma(i)}$ for all $i$.
\end{lemma}
\begin{proof}
By induction and part (b) of Lemma~\ref{lem:extendAIMC} we show that there exists a unique homomorphism $\rho_{\sigma}^{(k)}:R[x_1,x_2,\dots,x_k]\to R[x_1,x_2,\dots,x_n]$ with $\rho_{\sigma}^{(k)}(a)=a$ for all $a\in R$ and $\rho_{\sigma}^{(k)}(x_i)=x_{\sigma(i)}$ for $i=1,2,\dots,k$. Now we take $\rho_{\sigma}=\rho_{\sigma}^{(n)}$. From the uniqueness follows that $\rho_{\tau}\rho_{\sigma}=\rho_{\tau\sigma}$ and if $1\in S_n$ is the identity then $\rho_1$ is the identity. In particular, we have $\rho_{\sigma}\rho_{\sigma^{-1}}=\rho_{\sigma^{-1}}\rho_\sigma=\rho_1$ is the indentity, so $\rho_{\sigma}$ is an automorphism.
\end{proof}
\begin{proposition}\label{prop:Universal}
    Suppose $\phi:R\to S$ is a homomorphism between convex semirings with $\phi^{-1}({\bf 0})=\{{\bf 0}\}$. For given $y_1,y_2,\dots,y_n\in S$ there exists a unique homomorphism $\widehat{\phi}:R[x_1,x_2,\dots,x_n]\to S$ with $\widehat{\phi}(a)=\phi(a)$ for all $a\in R$ and $\widehat{\phi}(x_i)=y_i$ for all $i$.
\end{proposition}
\begin{proof}
Suppose there is an integer $k$ such that $y_1,y_2,\dots,y_k$ are nonzero
and $y_{k+1}=y_{k+2}=\cdots=y_{n}={\bf 0}$. Then the proposition follows from Lemma~\ref{lem:extendAIMC} and induction. In the general case, there exists a permutation $\sigma\in S_n$ and an integer $k$ such that $y_{\sigma(1)}, y_{\sigma(2)},\dots,y_{\sigma(k)}$ are nonzero 
and $y_{\sigma(k+1)}=y_{\sigma(k+2)}=\cdots=y_{\sigma(n)}={\bf 0}$. 
There exists an extension $\widehat{\phi}:R[x_1,x_2,\dots,x_n]\to S$
with $\widehat{\phi}(x_i)=y_{\sigma(i)}$. If we replace $\widehat{\phi}$ with $\widehat{\phi}\circ \rho_{\sigma^{-1}}$ (where $\rho_{\sigma^{-1}}$ is defined in Lemma~\ref{lem:permutePoly}) then we get $\widehat{\phi}(x_i)=y_i$ for all $i$ and $\widehat{\phi}(a)=a$ for all $a\in R$.
The uniqueness is clear.
\end{proof}
\subsection{The convex semiring of convex sets}
Let $\PP_n$ be the set of all compact convex subsets of $\R^n$ (including the empty set). For a subset $A\subseteq \R^n$ we write $\conv(A)$ for the convex hull of $A$. We define addition in $\PP_n$ by 
$$A\oplus B=\conv(A\cup B).$$ 
Multiplication  is given by the Minkowski sum:
$$
A\odot B=A+B=\{a+b\mid a\in A,b\in B\}.
$$
The identity for addition is ${\bf 0}:=\emptyset$ and the identity for multiplication is ${\bf 1}:=\{0\}\subseteq \R^n$. Now $(\PP_n,\oplus,\odot)$ is a semiring. For all $A\in \PP_n$ we have $A\oplus A=\conv(A\cup A)=A$, so $\PP_n$ is idempotent. The cancellation property holds for the Minkowski sum of convex compact subsets of $\R^n$, so $\PP_n$ is also cancellative. So the semiring $\PP_n$ is convex.

\begin{theorem}
    The semiring $\B[x_1,x_2,\dots,x_n]$ is isomorphic to the the sub-semiring of $\PP_n$ consisting of all convex hulls of finite subsets of $\N^n$.
\end{theorem}
\begin{proof}
There is a unique homomorphism $\phi:\B\{x_1,x_2,\dots,x_n\}\to \PP_n$ with $\phi(x_i)=\{{\bf e}_i\}$ where ${\bf e}_i$ is the $i$-th basis vector in $\R^n$. For $\alpha\in \N^n$ we have $$\textstyle \phi(x^{\alpha})=\phi\big(\bigodot_{i=1}^n x_i^{\alpha_i}\big)=
\bigodot_{i=1}^n \{{\bf e}_i\}^{\alpha_i}=\sum_{i=1}^n \{\alpha_i {\bf e}_i\}=\big\{\sum_{i=1}^n \alpha_i{\bf e}_i\big\}=\{\alpha\}.$$
For a finite subset $S\subseteq \N^n$ we get
$$
\textstyle
\phi\big(\bigoplus_{\alpha\in S} x^\alpha\big)=\bigoplus_{\alpha\in S} \{\alpha\}=\conv(S).
$$
Because $\PP_n$ is convex, the map $\phi$ factors through a homomorphism $\overline{\phi}:\B[x_1,x_2,\dots,x_n]\to \PP_n$.
The image of $\phi$ and of $\overline{\phi}$ is exactly the set of all convex hulls of finite subsets of $\N^n$. We will show that $\overline{\phi}$ is injective.
Suppose that $f=\bigoplus _{\alpha\in S}x^\alpha, g=\bigoplus_{\beta\in T}x^\beta\in \B[x_1,x_2,\dots,x_n]$
and $\overline{\phi}(f)=\overline{\phi}(g)$.
Then we have $\conv(S)=\conv(T)$.
If $\beta\in T$ then $\beta\in \conv(S)$ and therefore
$x^\beta\leq \bigoplus_{\alpha\in S}x^\alpha=f$ by Lemma~\ref{lem:convex}.
This is true for all $\beta\in T$, so
$g=\bigoplus _{\beta\in T} x^\beta\leq f$. Similarly, we can show that  $f\leq g$, so $f=g$.

\end{proof}

\begin{definition}\label{def:NP}
  Suppose that $R$ is a convex semiring, and define $\phi:R\to \B$ by $\phi({\bf 0})={\bf 0}$ and $\phi(a)={\bf 1}$ for all $a\in R\setminus\{{\bf 0}\}$. Then there exists a unique homomorphism $\Pi:R[x_1,x_2,\dots,x_n]\to \B[x_1,x_2,\dots,x_n]$
  with $\Pi(a)=\phi(a)$ for all $a\in R$ and $\Pi(x_i)=x_i$ for all $i$. For an element $f\in R[x_1,x_2,\dots,x_n]$,  $\Pi(f)\in \B[x_1,x_2,\dots,x_n]\subseteq \PP_n$ is the {\em Newton polytope} of $f$.
\end{definition}

\section{Tropical invariants}
\subsection{The transfer map}
Suppose that $R$ is an idempotent semiring and $G$ is a finite group acting on $R$ by automorphisms. The invariant semiring is $R^G=\{f\in R\mid \forall g\in G\, g\cdot f=f\}$. 
\begin{definition}
    The transfer map $\Tr:R\to R^G$ is defined by
    $$
    \Tr_G(a)=\bigoplus_{g\in G} g\cdot a.
    $$
\end{definition}
One may think of $\Tr_G$ as a Reynolds operator in Invariant Theory or the transfer map in Modular Invariant Theory.
Some obvious properties of $\Tr_G$ are 
$\Tr_G(a)=a$, $\Tr_G(ab)=a\Tr_G(b)$ and $
\Tr_G(b\oplus c)=\Tr_G(b)\oplus\Tr_G(c)$ for $a\in R^G$ and $b,c\in R$. 
\begin{lemma}\label{lem:surjectiveInvariants}
If $\phi:R\to S$ is a $G$-equivariant homomorphism between additive idempotent semirings, then $\phi(R^G)=\phi(R)^G$.
\end{lemma}
\begin{proof}
For $b\in R$ we have
$$
\phi(\Tr_G(b))=\phi\Big(\bigoplus_{g\in G} g\cdot b\Big)=\bigoplus_{g\in G} g\cdot \phi(b)=
\Tr_G(\phi(b))\in S^G.
$$
In particular, if $b\in R^G$ then we have $\phi(b)=\phi(\Tr_G(b))\in S^G$, so $\phi(R^G)\subseteq S^G\cap \phi(R)=\phi(R)^G$.
For every $a\in \phi(R)^G$ there exists a $b\in R$ with $\phi(b)=a$ and $a=\Tr_G(a)=\Tr_G(\phi(b))=\phi(\Tr_G(b))$. Since $\Tr_G(b)\in R^G$ this shows that 
 $\phi(R)^G\subseteq \phi(R^G)$.
\end{proof}
If $S$ and $T$ are sub-semirings of a semiring $R$,
then $ST$ denotes the smallest sub-semiring of $R$ containing $S$ and $T$. The semiring $ST$ is the set of all elements of the form $\bigoplus_{i=1}^r a_ib_i$ with $a_i\in S$ and $b_i\in T$ for all $i$.

% \begin{lemma}
% Suppose that $G$ is a finite group acting on an idempotent semiring $R$. Assume that $S$ and $T$ are sub-semirings of $R$ with $S\subseteq R^G$,  $G\cdot T\subseteq T$ and $ST=R$. Then $R$ is generated by $S$ and $T^G$.
% \end{lemma}
% \begin{proof}
% We have $R^G=\Tr_G(ST)=S\Tr_G(T)=ST^G$.
% \end{proof}

\subsection{Symmetric group invariants}
Suppose that $R$ is a convex semiring and the symmetric group $S_n$ acts on the convex polynomial ring $R[x_1,x_2,\dots,x_n]$ by permuting the variables. We define the (tropical) elementary symmetric functions $e_1,e_2,\dots,e_n$ by
$$
e_k=\bigoplus_{1\leq i_1<i_2<\cdots<i_k\leq n} x_{i_1}\odot x_{i_2}\odot \cdots\odot x_{i_k}.
$$

\begin{theorem}[Rado's Theorem]\label{theo:convexhull}
Suppose $\alpha=(\alpha_1,\alpha_2,\dots,\alpha_n), \beta=(\beta_1,\beta_2,\dots,\beta_n)\in \R^n$ with 
$\alpha_1\geq \alpha_2\geq \cdots \geq \alpha_n$ and $\beta_1\geq \beta_2\geq \cdots \geq \beta_n$.
Then $\beta$ lies in the convex hull of $\sigma(\alpha)$, $\sigma\in S_n$
if and only if $\alpha_1+\alpha_2+\cdots+\alpha_n=\beta_1+\beta_2+\cdots+ \beta_n$ and $ \alpha_1+\alpha_2+\cdots +\alpha_k\geq \beta_1+\beta_2+\cdots+\beta_k$
for $k=1,2,\dots,n-1$.
\end{theorem}
For the proof, see for example~\cite[VI, 2.3]{barvinok2002course}.

\begin{lemma}\label{lem:ProductTransfers}
Suppose $\alpha=(\alpha_1,\alpha_2,\dots,\alpha_n),\beta=(\beta_1,\beta_2,\dots,\beta_n)\in \N^n$ with
$\alpha_1\geq \alpha_2\geq \cdots\geq \alpha_n$
and $\beta_1\geq \beta_2\geq \cdots \geq \beta_n$. 
In $R[x_1,x_2,\dots,x_n]$  we have
$$
\Tr_{S_n}(x^\alpha)\Tr_{S_n}(x^\beta)=\Tr_{S_n}(x^{\alpha+\beta}).
$$
\end{lemma}
\begin{proof}
We have
\begin{equation}\label{eq:sigmatau}
\Tr_{S_n}(x^\alpha)\Tr_{S_n}
(x^\beta)=\bigoplus_{\sigma\in S_n}\bigoplus_{\tau\in S_n} x^{\sigma(\alpha)+\tau(\beta)}\geq \bigoplus_{\sigma\in S_n} x^{\sigma(\alpha+\beta)}=\Tr_G(x^{\alpha+\beta}).
\end{equation}
Suppose that $\gamma=\sigma(\alpha)+\tau(\beta)$ for some $\sigma,\tau\in S_n$.
First assume that $\gamma_1\geq \gamma_2\geq \cdots \geq \gamma_n$. 
Then we have $\sum_{i=1}^n(\alpha_i+\beta_i)=\sum_{i=1}^n\gamma_i$ and $\sum_{i=1}^k(\alpha_i+\beta_i)\geq \sum_{i=1}^k \gamma_i$.
By Theorem~\ref{theo:convexhull}, $\gamma$ lies in the convex hull of all $\sigma(\alpha+\beta)$, $\sigma\in S_n$, 
so we have $x^\gamma\leq \Tr_{S_n}(x^{\alpha+\beta})$.
If $\gamma$ is not weakly decreasing, then $\lambda(\gamma)$ is weakly decreasing for some $\lambda\in S_n$ and $x^{\lambda(\gamma)}\leq \Tr_{S_n}(x^{\alpha+\beta})$. By symmetry we have
$x^\gamma=\lambda^{-1}\cdot x^{\lambda(\gamma)}\leq \Tr_{S_n}(x^{\alpha+\beta})$.
Because $x^{\sigma(\alpha)+\tau(\beta)}\leq \Tr_{S_n}(x^{\alpha+\beta})$ for all $\sigma,\tau\in S_n$ we get $\Tr_{S_n}(x^\alpha)\Tr_{S_n}(x^\beta)\leq \Tr_{S_n}(x^{\alpha+\beta})$. Together with~(\ref{eq:sigmatau}) we get equality.
\end{proof}

In \cite{kalivsnik2019} a semiring $R$ is called {\em fully elementary}
if the semiring of $S_n$-invariant polynomials in $n$ variables over a semiring $R$ (viewed as functions $R^n\to R$) is generated by elementary symmetric polynomials. So any commutative ring (with identity) is fully elementary.
Carlsson and Kali\v{s}nik Verov\v{s}ek showed in~\cite{carlsson2016} that the tropical semifield $\T$ is fully elementary. 
Kali{\v{s}}nik and Le{\v{s}}nik analyse in~\cite{kalivsnik2019}
which semirings are fully elementary. They show (see~\cite[Corollary 4.7]{kalivsnik2019}) that an additively idempotent semiring is fully elementary if and only if it is {\em  Frobenius}, i.e., it has the property 
$(a\oplus b)^n=a^n\oplus b^n$ for all $a,b\in R$ and all $n\geq 1$. In this paper we restrict ourselves to semirings $R$ that are convex, i.e., additively idempotent and multiplicatively cancellative. Such semirings are also Frobenius by Lemma~\ref{lem:TropicalishProperties}(e).
In \cite{kalivsnik2019} the authors make a distinction between syntactic polynomial expressions which correspond to elements of $R\{x_1,x_2,\dots,x_n\}$ in our notation, and polynomial functions $R^n\to R$ that can be represented by such syntactic polynomials. 
Both notions of ``polynomial'' are different from our notion of a polynomial as an element in the convex polynomial ring $R[x_1,x_2,\dots,x_n]$. So even though a convex semiring is additively idempotent and Frobenius, the following theorem does not immediately follow from~\cite{kalivsnik2019}:

\begin{theorem}\label{thm:Sninvariants}
    If $R$ is a convex semiring, then the invariant semiring $R[x_1,x_2,\dots,x_n]^{S_n}$ is generated over $R$ by $e_1,e_2,\dots,e_n$.
\end{theorem}
\begin{proof}
    Let $R[e_1,e_2,\dots,e_n]$ be the sub-semiring generated over $R$ by $e_1,e_2,\dots,e_n$. As an $R$-module, $R[x_1,x_2,\dots,x_n]^G$ is generated by all  $\Tr_{S_n}(x^{\gamma})$ where $\gamma=(\gamma_1,\gamma_2,\dots,\gamma_n)\in \N^n$ satisfies $\alpha_1\geq \alpha_2\geq \cdots \geq \alpha_n\geq 0$.
    By induction on $|\gamma|=\gamma_1+\gamma_2+\cdots+\gamma_n$
    we prove that $\Tr(x^{\gamma})\in R[x_1,x_2,\dots,x_n]$.
    The case $|\gamma|=0$ is clear. Suppose $|\gamma|> 0$.
    If $\gamma_1=\gamma_2=\cdots=\alpha_n=k>0$ then we have $\Tr(x^{\gamma})=e_n^{k}$. Otherwise, there exists an $i$ with $\gamma_i>\gamma_{i+1}$. Define $\beta=(\beta_1,\beta_2,\dots,\beta_n)$ by $\beta_1=\beta_2=\cdots=\beta_i=1$ and $\beta_{i+1}=\beta_{i+2}=\cdots=\beta_{n}=0$ and let $\alpha=(\alpha_1,\alpha_2,\dots,\alpha_n)=\gamma-\beta$.
    Then we have $\alpha_1\geq \alpha_2\geq \cdots \geq \alpha_n\geq 0$, and 
     Lemma~\ref{lem:ProductTransfers} implies that $\Tr(x^\gamma)=\Tr(x^\alpha)\Tr(x^\beta)=\Tr(x^\alpha)e_i$.
    By induction we have $\Tr(x^{\alpha})\in R[e_1,e_2,\dots,e_n]$, so we get $\Tr(x^\gamma)\in R[x_1,x_2,\dots,x_n]$.
\end{proof}
If $R$ is a commutative ring, then the invariant ring $R[x_1,x_2,\dots,x_n]^{S_n}$ is isomorphic to the polynomial ring $R[x_1,x_2,\dots,x_n]$. This is not true in the tropical case:

\begin{proposition}
For $n\geq 2$, $R[x_1,x_2,\dots,x_n]^{S_n}=R[e_1,e_2,\dots,e_n]$ is not isomorphic to $R[x_1,x_2,\dots,x_n]$ as a semiring over $R$.  
\end{proposition}
\begin{proof}
    Suppose $\phi:R[x_1,x_2,\dots,x_n]\to R[e_1,e_2,\dots,e_n]$
    is an isomorphism over $R$. Let $y_i=\phi(x_i)\in R[e_1,e_2,\dots,e_n]\subseteq R[x_1,x_2,\dots,x_n]$, so $R[e_1,e_2,\dots,e_n]=R[y_1,y_2,\dots,y_n]$.
    We can write $e_1=\bigoplus_{\alpha\in S} c_\alpha y^{\alpha}$, where $S\subseteq \N^n$ is a finite subset and $c_\alpha\in R\setminus \{0\}$ for all $\alpha\in S$. 
We now use the map $\Pi:R[x_1,x_2,\dots,x_n]\to \PP_n$ from Definition~\ref{def:NP} that sends
a polynomial $f$ to its Newton polytope $\Pi(f)\in \PP_n$.
We have
$$\{(1,0,\dots,0)\}=\Pi(x_1)\subseteq \conv\left(\bigcup_{\alpha\in S} \Pi(y^{\alpha})\right)
$$
There must be a positive integer $k$ and an element $\alpha=(\alpha_1,\alpha_2,\dots,\alpha_n)\in S$ with
$$
(k,0,0,\dots,0)\in \Pi(y^\alpha)=\bigoplus_{i=1}^n \alpha_i\Pi(y_i).
$$
From this follows that there exists a positive integer $\ell$ and an $i\in \{1,2,\dots,n\}$ such that
$(\ell,0,0,\dots,0)\in \Pi(y_i)$.
Without loss of generality, assume that $i=1$.
We have 
$dx_1^\ell\leq y_1$ for some $d\in R\setminus\{0\}$. It follows that $de_1^\ell=d\Tr(x_1^\ell)\leq \Tr(y_1)=y_1$.
    Suppose that $y_2\in R[x_1,x_2,\dots,x_n]$ has degree $\leq m$ as a polynomial in $x_1,x_2,\dots,x_n$. 
    Then  we have 
$$y_2\leq c({\bf 1}\oplus x_1\oplus x_2\oplus \cdots\oplus x_n)^m=c({\bf 1}\oplus e_1)^m\leq c({\bf 1}\oplus e_1)^{\ell m}=c\oplus ce_1^{\ell m}
$$ 
    for some $c\in R\setminus\{0\}$.  
It follows that 
$$d^m y_2\leq cd^m\oplus cd^m e_1^{\ell m}=cd^m\oplus cy_1^m.
$$
    Because $\phi$ is an isomorphism, we have $d^m x_2\leq cd^m\oplus cx_1^m$.
    This is a contradiction.
\end{proof}
\begin{lemma}\label{lem:productgroup}
    Suppose $G$ and $H$ are groups that act on the convex polynomial rings $S=R[x_1,x_2,\dots,x_n]$ and $T=R[y_1,y_2,\dots,y_m]$ respectively by permuting the variables. Now $G\times H$ acts on the semiring  $A=R[x_1,x_2,\dots,x_n,y_1,y_2,\dots,y_n]$.
    Then 
    $A^{G\times H}$ is generated by $S^G$ and $T^H$. 
\end{lemma}
\begin{proof}
The group $H$ acts trivially on $S$ and $G$ acts trivially of $T$.  We have 
$$\Tr_{G\times H}(ST)=\Tr_G(\Tr_H(ST))=
\Tr_G(S\Tr_H(T))=\Tr_G(ST^H)=\Tr_G(S)T^H=S^GT^H.
$$
\end{proof}

% If $f\in R[x_1,x_2,\dots,x_n,y_1,y_2,\dots,y_m]^{G\times H}$ then we can write it in the form $$f=\sum_{i=1}^s a_i(x_1,\dots,x_n)b_i(y_1,\dots,y_m).$$
% Then 
% $$
% f=\Tr_{G\times H}(f)=\sum_{i=1}^r \Tr_G(a_i(x_1,\dots,x_n))\Tr_H(b_i(y_1,\dots,y_m))
% $$
% lies in the semiring generated by $R[x_1,x_2,\dots,x_n]^G$ and $R[y_1,y_2,\dots,y_n]^H$. 
% \end{proof}
If a subgroup $G\subseteq S_n$ is generated by $2$-cycles, then it is a product of symmetric groups. To be more precise, suppose that $G$ is generated by the 2-cycles $(i_k,j_k)$, $k=1,2,\dots,r$. One can draw a graph on the vertex set $\{1,2,\dots,n\}$ by drawing an edge between $i_k$ and $j_k$. Let $A_1,A_2,\dots,A_\ell\subseteq \{1,2,\dots,n\}$ be the connected components of this graph. Then we have $G=\Sym(A_1)\times \Sym(A_2)\times \cdots \Sym(A_\ell)$.
\begin{corollary}\label{cor:FinitelyGenerated}
    If $G\subseteq S_n$ is generated by $2$-cycles, then $R[x_1,x_2,\dots,x_n]^G$ is finitely generated over $R$.
\end{corollary}
\begin{proof}
    This follows by induction on $n$ using Theorem~\ref{thm:Sninvariants} and Lemma~\ref{lem:productgroup}.
\end{proof}
We will see in the next section that the converse of the Corollary is also true.

\subsection{Permutation group invariants}

\begin{theorem}\label{theo:NotFinitelyGenerated}
    Suppose that $G\subseteq S_n$ is a subgroup that is not generated by $2$-cycles and $R$ is a convex semiring. Then $R[x_1,x_2,\dots,x_n]^G$ is not finitely generated over $R$.
\end{theorem}

\begin{proof}
Consider the map $\Pi:R[x_1,x_2,\dots,x_n]\to \B[x_1,x_2,\dots,x_n]\subset \PP_n$ from Definition~\ref{def:NP} that maps
$f\in R[x_1,x_2,\dots,x_n]$ to its Newton polytope. 
By Lemma~\ref{lem:surjectiveInvariants}, we 
 we have $\B[x_1,x_2,\dots,x_n]^G=\Pi(R[x_1,x_2,\dots,x_n]^G)$.
If $R[x_1,x_2,\dots,x_n]^G$ is finitely generated over $R$, then $\B[x_1,x_2,\dots,x_n]^G=\Pi(R[x_1,x_2,\dots,x_n]^G)$ is finitely generated over $\Pi(R)=\B$. So we can reduce the theorem to the case where $R=\B$. We view
$\B[x_1,x_2,\dots,x_n]$ as a sub-semiring of $\PP_n$ consisting of convex hulls of finite subsets of $\N^n$,
where the monomial $x^\alpha$ is dentified with the convex set $\{\alpha\}\subseteq \R^n$. 
For $\alpha\in \N^n$,  
$\Tr_G(\{\alpha\})$ is the convex hull of all $\sigma(\alpha)$, $\sigma\in G$. Define
$$\Delta=\{(\alpha_1,\alpha_2,\dots,\alpha_n)\in \N^n \mid \alpha_1>\alpha_2>\cdots>\alpha_n\geq 0\}.
$$
$N$ be the group generated by all $2$-cycles in $G$ and assume $N\neq G$.
Then $N$ is a normal subgroup of $G$.\\[5pt]
{\bf Claim 1:} If $\alpha\in \Delta$ then there exists a group element $\sigma\in G\setminus N$ such that there is an edge between $\alpha$ and $\sigma(\alpha)$ in the  polytope $A=\Tr_G(\{\alpha\})$.\\[5pt]
Because the $1$-skeleton of $A$ is connected, 
there must be an edge of $A$ between a vertex 
$\tau(\alpha)$ with $\tau\in N$
and another vertex $\lambda(\alpha)$ with $\lambda\in G\setminus N$.
If we set $\sigma=\tau^{-1}\lambda\in G\setminus N$,
then by the $G$-symmetry there is also an edge between
$\alpha$ and $\tau^{-1}(\lambda(\alpha))=\sigma(\alpha)$. This proves Claim~1.\\[5pt]
{\bf Claim 2:} There exists an infinite subset $S\subseteq \Delta\subseteq\N^n$  such that for all $\sigma\in G\setminus N$ and for every $\alpha,\beta\in S$
with $\alpha\neq \beta$, the two vectors $\alpha-\sigma(\alpha)$ and $\beta-\sigma(\beta)$ are linearly independent.\\[5pt]
We take $S=\{\alpha^{(1)},\alpha^{(2)},\dots\}$ where the  sequence of vectors $\alpha^{(1)},\alpha^{(2)},\alpha^{(3)},\dots\in \N^n$ is constructed as follows as follows.
Let $$T=\bigcup_{\sigma \in G\setminus N} \ker(1-\sigma)\subseteq \R^n.$$
Every element $\sigma\in G\setminus N$ is not a $2$-cycle, so $\ker(1-\sigma)$ has dimension $\leq n-2$.
We choose $\alpha^{(1)}\in \Delta\setminus T$. 
Suppose we have already chosen $\alpha^{(1)},\alpha^{(2)},\dots,\alpha^{(i-1)}$.
Consider the set $U^{(i)}=\bigcup_{j=1}^{i-1} \R \alpha^{(j)}+T$.
Then $U^{(i)}$ is a finite union of subspaces of dimension $\leq n-1$. The set $\Delta$ is not contained in a finite union of subspaces of dimension $\leq n-1$ so we can choose $\alpha^{(i)}\in \Delta\setminus U^{(i)}$.
Suppose that $\alpha^{(i)}-\sigma(\alpha^{(i)})$ and $\alpha^{(j)}-\sigma(\alpha^{(j)})$
are linearly dependent for some $\sigma\in G\setminus N$ and $j<i$.
Say, $\alpha^{(i)}-\sigma(\alpha^{(i)})=\lambda(\alpha^{(j)}-\sigma(\alpha^{(j)}))$. Then we have $\alpha^{(i)}-\lambda\alpha^{(j)}\in \ker(1-\sigma)$, so $\alpha^{(i)}\in \R\alpha^{(j)}+T\subseteq U^{(i)}$. Contradiction. So the set $T=\{\alpha^{(1)},\alpha^{(2)},\dots\}$ has the desired properties and Claim~2 has been proved.\\[5pt]

Suppose $P_1,P_2,\dots,P_r\in \B[x_1,x_2,\dots,x_n]$
generate $\B[x_1,x_2,\dots,x_n]^G$.
Let
$\alpha\in \Delta$.
If $\Tr_G(\{\alpha\})=B\oplus C=\conv(B\cup C)$ with $B,C\in \B[x_1,x_2,\dots,x_n]^G$, then $\alpha$ is vertex of $B$ or of $C$.
This proves that $\Tr_G(\{\alpha\})\subseteq  B$
or $\Tr_G(\{\alpha\})\subseteq C$ and therefore
$\Tr_G(\{\alpha\})=B$ or $\Tr_G(\{\alpha\})=C$.
Because $P_1,P_2,\dots,P_r$ generate $\B[x_1,x_2,\dots,x_n]^G$,
we can write 
$\Tr(\{\alpha\})=M_1\oplus M_2\oplus\cdots\oplus M_s$,
where $M_1,M_2,\dots,M_s$ are monomials in $P_1,P_2,\dots,P_r$. By induction on $s$ we see that $\Tr(\{\alpha\})$ is equal to $M_j$ for some $j$.
We can write 
\begin{equation}\label{eq:PolytopeSum}
\Tr(\{\alpha\})=P_1^{\odot \beta_1}\odot P_2^{\odot\beta_2}\odot \cdots\odot P_r^{\odot\beta_r}=\beta_1 P_1+\beta_2P_2+\cdots +\beta_rP_r
\end{equation}
for some $\beta_1,\beta_2,\dots,\beta_r\in \N$.

Let $D$ be the set of all unit vectors $(\alpha-\beta)/\|\alpha-\beta\|$ where $\alpha$ and $\beta$  are adjacent vertices of a polytope $P_j$ for some $j$.
For $\alpha\in S$,  $\Tr_G(\{\alpha\})$ has an edge between $\alpha$ and $\sigma_\alpha(\alpha)$ for some permutation $\sigma_\alpha\in G\setminus N$ by Claim~1.
From~(\ref{eq:PolytopeSum}) follows that $\alpha-\sigma_\alpha(\alpha)$ is parallel to an edge of $P_j$ for some $j$, so  $(\alpha-\sigma_\alpha(\alpha))/\|\alpha-\sigma_\alpha(\alpha)\|\in D$.
There exists a $\tau$ such that $\sigma_\alpha=\tau$ for infinitely many $\alpha\in S$. So we get that $(\alpha-\tau(\alpha))/\|\alpha-\tau(\alpha)\|\in D$
for infinitely many $\alpha\in S$. By Claim~2,  all these vectors are distinct. On the other hand, $D$ is finite so we get a contradiction. We conclude that $\B[x_1,x_2,\dots,x_n]^G$ is not finitely generated.
\end{proof}

\section{Tropical rational invariants}\label{sec:Rational}
\subsection{Tropical function semifields and L\"uroth's problem}

If $\F$ is a field, then any intermediate field
$\F\subseteq L\subseteq \F(x_1,x_2,\dots,x_n)$ is finitely generated over $\F$. An analog of this property for idempotent semifields is not true as the following proposition shows:
\begin{proposition}
    Let $P_i={\bf 1}\oplus x_1x_2^i=\conv((0,0),(1,i))\in \B[x_1,x_2]\subset \B(x_1,x_2)$. Then the semifield
    $L=\B(P_1,P_2,P_3,\dots)$ is not finitely generated.
\end{proposition}
\begin{proof}
    For a non-constant polynomial $A\in \B[x_1,x_2]\subset \PP_2$
    we define $\slope(A)$ as the smallest real number $r$ for which the convex set $A$ lies below the line $y=rx$. If $A$ contains a point $(0,a)$ with $a>0$ then $\slope(A)$ is equal to $\infty$. We have $\slope(A\oplus B)=\slope(\conv(A\cup B))=\max\{\slope(A),\slope(B)\}$ and $\slope(A\odot B)=\slope(A+B)=\max\{\slope(A),\slope(B)\}$, $\slope(A\oplus {\bf 1})=\slope(A)$ if $A$ and $B$ are nonempty. We have $\slope(P_i)=i$ for all $i$.  By induction it is easy to see that $\slope(A)\leq k$ for all nonconstant $A\in \B[P_1,P_2,\dots,P_k]$.

    Suppose that $L$ is finitely generated by $Q_1,Q_2,\dots,Q_s\in \B(x_1,x_2)$. Then there exists a positive integer $k$ such that
   $Q_i\in \B(P_1,P_2,\dots,P_k)$ for $i=1,2,\dots,s$. It follows that
 $$L=\B(Q_1,Q_2,\dots,Q_s)\subseteq \B(P_1,P_2,\dots,P_k)\subseteq L.$$
Therefore $L=\B(P_1,P_2,\dots,P_k)$.
Because $P_{k+1}\in L=\B(P_1,P_2,\dots,P_k)$ there exist nonconstant $A,B\in \B[P_1,P_2,\dots,P_k]$ with $P_{k+1}=A\oslash B$. 
Since $A=P_{k+1}\odot B$ we get
$$
k\geq\slope(A)=\max\{\slope(P_{k+1}),\slope(B)\}=
\max\{k+1,\slope(B)\}=k+1.
$$
Contradiction! Hence $L$ is not finitely generated over $\B$.
\end{proof}

\begin{proposition}
    Define $h_i={\bf 1}\oplus (-1)x^i\in \T(x)$, for $i=1,2,\dots$. Then the semifield $L=\T(h_1,h_2,\dots)$ is not finitely generated. 
    \end{proposition}
    \begin{proof}
We will view nonzero elements of $\T(x)$ as piecewise linear functions $\R\to \R$.
Then $h_i$ is identified with the function $h_i(x)=\max\{0,ix-1\}$. 
Now $h_i$ is constant on the interval $[0,1/i]$. If $h_{k+1}\in \T(h_1,h_2,\dots,h_k)$, then 
    Then $h_{k+1}$ can be obtained
    from $h_1,h_2,\dots,h_k$ by taking sums and $\max$. This shows that $h_{k+1}$ is constant on the interval $[0,1/k]$. However, $h_{k+1}(0)=0$ and $h_{k+1}(1/k)=1/k$. Contradiction.
    This proves that $L$ is not finitely generated.
    \end{proof}

Considering the results above, it is a natural question whether subfields of $\B(x)$ are finitely generated. We will show that this is true and that we have an analog of L\"uroth's theorem. Let $\B(x)^\times=\B(x)\setminus\{{\bf 0}\}$. 
\begin{lemma}
    The abelian multiplicative group $\B(x)^\times$ is freely generated by $x$ and ${\bf 1}\oplus x$.
\end{lemma}
\begin{proof}
    Nonzero elements of $\B[x]$ are of the form $x^\alpha\oplus x^\beta=x^\alpha\odot (1\oplus x)^{\beta-\alpha}$ with $0\leq \alpha\leq \beta$.
    Since elements of $\B(x)^\times$ are quotients of nonzero elements in $\B[x]$, we see that $x$ and ${\bf 1}\oplus x$ generate the group $\B(x)^\times$. Suppose $x^\alpha\odot ({\bf 1}\oplus x)^\beta={\bf 1}$
    for some $\alpha,\beta\in \Z$. By replacing $\alpha,\beta$ with $-\alpha,-\beta$ respectively we may assume that $\alpha\geq 0$. If $\beta\geq 0$ then $x^\alpha\odot ({\bf 1}\oplus x)^\beta={\bf 1}$
    has a nonzero constant term, so $\alpha=0$ and it follows that $\beta=0$ as well. If $\beta\leq 0$, then we have $x^\alpha=({\bf 1}\oplus x)^{\beta}$. Again looking at the constant term it follows that $\alpha=\beta=0$. This shows that $x$ and ${\bf 1}\oplus x$ are free abelian group generators.
\end{proof}
\begin{lemma} 
If $0\leq \alpha\leq \beta$, then $x^{-\alpha}\oplus {\bf 1},{\bf 1}\oplus x^\beta\in \B(x^\alpha\oplus x^\beta)$.
\end{lemma}
\begin{proof}
We have $1\oplus x^\beta=1\oplus (x^\alpha\oplus x^\beta)\in \B(x^\alpha\oplus x^\beta)$ and 
$$
x^{-\alpha}\oplus {\bf 1}=\frac{{\bf 1}\oplus x^\beta}{x^\alpha\oplus x^\beta}\in \B(x^\alpha\oplus x^\beta).
$$
\end{proof}
\begin{lemma}
    Suppose $f\in \B(x)^\times$. Then $f$ or $f^{-1}$ is of the form $x^\alpha\oplus x^\beta$
    with $\alpha\leq \beta$.
\end{lemma}
\begin{proof}
We can write $f=x^\gamma ({\bf 1}\oplus x)^\delta$ with $\gamma,\delta\in \Z$. If $\delta\geq 0$
then $f=x^\gamma\oplus x^{\gamma+\delta}$.
Otherwise $1\oslash f=x^{-\gamma} ({\bf 1}\oplus x)^{-\delta}=x^{-\gamma}\oplus x^{-\gamma-\delta}$. 
\end{proof}
\begin{proposition}
Suppose $L\subseteq \B(x)$ is a sub-semifield.
Then exactly one of the following statements is true:
\begin{enumerate}
    \item $L=\B(x^\alpha\oplus x^\beta)$ with $\alpha\leq 0\leq \beta$ and $L^\times$ consists of all elements $x^{n\alpha}\oplus x^{n\beta}$ with $n\geq 0$ and their inverses;
    \item $\L=\B(x^\alpha\oplus x^\beta)$ with $0<\alpha\leq \beta$ and $L^\times$ consists of all elements of the form $x^{n\alpha}\oplus x^{m\beta}$ with $n,m\in \Z$
    and $n\alpha\leq m\beta$ and their inverses.
\end{enumerate}
\end{proposition}
\begin{proof}
Let $r$ be the rank of the abelian group $L^\times\subseteq \B(x)^\times$. Since $\B(x)^\times$ has rank $2$, we have $0\leq r\leq 2$. If $r=0$, then $L=\B$ and we are in case (1) with $\alpha=\beta=0$.

Suppose $r=1$.
Then $L^\times$ is as a group generated by one element $f\in L^\times$. Without loss of generality we may write $f=x^\alpha\oplus x^\beta$
with $\alpha\leq \beta$. If $\alpha,\beta$ are both positive, then ${\bf 1}\oplus x^{-\alpha}=x^{-\alpha}({\bf 1}\oplus x)^\alpha$ and ${\bf 1}\oplus x^{\beta}=({\bf 1}\oplus x)^\beta$
are independent which gives a contradiction.
Similarly, $\alpha$ and $\beta$ cannot both be negative. So we may assume that $\alpha\leq 0\leq \beta$. For $n\geq 0$ we have $f^n=x^{n\alpha}\oplus x^{n\beta}$. 
Note that every polynomial in $f$ of degree $n$ is equal to $f^n$.
We are in case (1). 

Suppose that $r=2$.
There exists a positive integer $\gamma$ with $x^\gamma\in L^\times$. Then we have $x^{-\gamma}\oplus {\bf 1},{\bf 1}\oplus x^\gamma\in L^\times$.
Let $\alpha$
and $\beta$ be the smallest positive integers
with $x^{-\alpha}\oplus {\bf 1},{\bf 1}\oplus x^\beta\in L^\times$.
Note that $x^{\alpha\beta}=({\bf 1}\oplus x^\beta)^\alpha(x^{-\alpha}\oplus {\bf 1})^\beta\in L^\times$.
If ${\bf 1}\oplus x^\gamma\in L^\times$ with $\gamma\geq 0$,
then one can write $\gamma=n\beta+\rho$
where $0\leq \rho<\beta$.
It follows that $({\bf 1}\oplus x^\gamma)\oslash({\bf 1}\oplus x^{\beta})^n={\bf 1}\oplus x^\rho$. By minimality of $\beta$, we get $\rho=0$ and $\gamma$ is divisible by $\beta$.
If ${\bf 1}\oplus x^\gamma\in L^\times$ with $\gamma\leq 0$,
then a similar argument shows that $\gamma$ is divisible by $\alpha$.
If $g\in L^\times$ then $g$ or $g^{-1}$ is of the form $x^\gamma\oplus x^\delta\in L^\times$ with $\gamma\leq \delta$.
For some integer $k$, $\gamma'=k\alpha\beta+\gamma$ and $\delta'=k\alpha\beta+\delta$ are positive.
We have $x^{\gamma'}\oplus x^{\delta'}=(x^{\alpha\beta})^k(x^\gamma\oplus x^\delta)\in L^\times$. It follows that $x^{-\gamma'}\oplus {\bf 1},{\bf 1}\oplus x^{\delta'}\in L^\times$.
So $\gamma'=\gamma+k\alpha\beta$ is divisible by $\alpha$ and $\delta'=\delta+k\alpha\beta$ is divisible by $\beta$. It follows that $\gamma$ is divisible by $\alpha$ and $\delta$ is divisible by $\beta$. If we write $\gamma=n\alpha$ and $\delta=m\beta$ with $m,n\in \Z$,
then we have $x^\gamma\oplus x^\delta=x^{n\alpha}\oplus x^{m\beta}$
with $n\alpha=\gamma\leq \delta=m\beta$.
On the other hand, if $n\alpha\leq m\beta$, then we have
$$({\bf 1}\oplus x^{-\alpha})^{-n}({\bf 1}\oplus x^\beta)^m=
x^{n\alpha}({\bf 1}\oplus x)^{-\alpha n}({\bf 1}\oplus x)^{\beta m}=x^{n\alpha}({\bf 1}\oplus x)^{\beta m-\alpha n}=x^{n\alpha}\oplus x^{m\beta}.
$$

\end{proof}
\subsection{Rational invariants}
Suppose $\F$ is an idempotent semifield. Since semifields are cancellative, $\F$ is convex.
Let $G\subseteq S_n$ be a subgroup. Then $G$ acts on $\F[x_1,x_2,\dots,x_n]$ and $\F(x_1,x_2,\dots,x_n)=Q(\F[x_1,x_2,\dots,x_n])$ by permuting the variables.
\begin{lemma}
We have $\F(x_1,x_2,\dots,x_n)^G=Q(\F[x_1,x_2,\dots,x_n]^G)$.
\end{lemma}
\begin{proof}
Suppose that $h\in \F(x_1,x_2,\dots,x_n)^G$. 
We can write $h=f\oslash g$ with $f,g\in \F[x_1,x_2,\dots,x_n]$. Let $u=\prod_{\sigma\in G} \sigma\cdot g$. Then $h=uh\oslash u$ and $uh,u\in \F[x_1,x_2,\dots,x_n]^G$. This proves that $\F(x_1,x_2,\dots,x_n)^G\subseteq Q(\F[x_1,x_2,\dots,x_n]^G)$. The opposite inclusion is obvious.
\end{proof}

For $\alpha=(\alpha_1,\alpha_2,\dots,\alpha_n)\in \R^n$, we define the $\ell_p$ norm by by
$\|\alpha\|_p=\big(\sum_{i=1}^n |\alpha_i|^p\big)^{1/p}$ and $\|\alpha\|_\infty=\max\{|\alpha_1|,|\alpha_2|,\dots,|\alpha_n|\}$.

\begin{theorem}\label{theo:RationalInvariants}
 Let $p_1,p_2,\dots,p_k$ be the first $k$ prime numbers where $k=|G|$. Then the invariant semifield $\F(x_1,x_2,\dots,x_n)^G$ is generated by all
$\Tr_G(x^\alpha)$ with  $\alpha\in \N^n$ and $\|\alpha\|_\infty<p_1p_2\cdots p_k$.
\end{theorem}
\begin{proof}
Suppose that $G=\{\sigma_1,\sigma_2,\dots,\sigma_k\}$.
Let $p_1=2<p_2<\cdots<p_k$ be the first $k$ prime numbers. We set $N=p_1p_2\cdots p_k$.
We define $S$ as the set of all invariants of the form
 $\Tr_G(x^\alpha)$ with $\alpha\in \N^n$
 and $\|\alpha\|_\infty<N$.
By induction on $\|\beta\|_\infty$ we show that
 $\Tr(x^\beta)$ can be written as a rational function in the elements of $S$ for all $\beta\in \N^n$. This is obvious when  $\|\beta\|_\infty< N$. Suppose that $M=\|\beta\|_\infty\geq N$.
 By the Chinese Remainder 
 Theorem there exists a vector $\gamma\in \N^n$ 
 with $\|\gamma\|_\infty<N$ and
 $\gamma+\sigma_i(\beta)\in p_i \Z^n$ for all $i$.
 Let $\delta_i=(\gamma+\sigma_i(\beta))/p_i$.
 We have $\|\delta_i\|_\infty<(N+M)/2<M$.
 
Now we have
\begin{multline*}
\Tr_G(x^\beta)\Tr_G(x^\gamma)=
\Big(\bigoplus_{i=1}^k x^{\sigma_i(\beta)}\Big)\Big(\bigoplus_{i=1}^k x^{\sigma_j(\gamma)}\Big)=
\bigoplus_{j=1}^k \bigoplus_{i=1}^k x^{\sigma_j(\gamma+\sigma_i(\beta))}=\\=
\bigoplus_{j=1}^k \bigoplus_{i=1}^kx^{\sigma_j(\delta_i)p_i}=
\bigoplus_{j=1}^k \Big( \bigoplus_{i=1}^k x^{\sigma_j(\delta_i)}\Big)^{p_i}=\bigoplus_{i=1}^k \Tr_G(x^{\delta_i})^{p_i}.
\end{multline*}
Now $\Tr_G(x^{\gamma})$ lies in $S$ because $\|\gamma\|_\infty<N$.
Also, by the induction hypothesis, $\Tr_G(x^{\delta_i})$ is a rational function in the elements of $S$ because $\|\delta_i\|_\infty<M$ for all $i$.
This proves that $\Tr_G(x^\beta)$ is a rational function in $S$.
\end{proof}
\begin{proof}[Proof of Theorem~\ref{theo:2}]
Note that $\|\alpha\|_1\leq n\|\alpha\|_\infty$. Now Theorem~\ref{theo:2} follows from Theorem~\ref{theo:RationalInvariants}, where $\F=\T$.
\end{proof}

\begin{corollary}
Suppose that
$L$ is an idempotent semifield that is finitely generated over a sub-semifield $\F$, and $G$ is a finite group that acts on $L$ by automorphisms over $\F$. Then $L^G$ is a finitely generated over $\F$.    
\end{corollary}
\begin{proof}
Suppose that $L=\F(y_1,y_2,\dots,y_n)$ for some $y_1,y_2,\dots,y_n\in L^\times$ we may assume that $\{y_1,y_2,\dots,y_n\}$ is a union of $G$-orbits.
 Note that $y_1,y_2,\dots,y_n$ may have nontrivial relations. By the universal property (Lemma~\ref{prop:Universal}), there exists a homomorphism $\phi$ from the tropical polynomial ring $\F[x_1,x_2,\dots,x_n]$
to $L$ such that $\phi(x_i)=y_i$ for all $i$
and $\phi(a)=a$ for all $\in \F$. Because $\phi(x_i)=y_i$ is nonzero for all $i$,
we have $\phi^{-1}({\bf 0})={\bf 0}$.
We can lift the $G$-action to $\F[x_1,x_2,\dots,x_n]$. Since $G$ acts by permuting the set $\{y_1,y_2,\dots,y_n\}$, 
an element $\sigma\in G$ sends $y_i$ to another generator $y_{\sigma(i)}$. We define the action of $G$ on $\F[x_1,x_2,\dots,x_n]$ by $\sigma\cdot x_i=x_{\sigma(i)}$. Now $\phi$ is $G$-equivariant.
By the universal property of the quotient semifield, $\phi$ extends to a homomorphism
$\F(x_1,x_2,\dots,x_n)\to L$. Now $\phi$ is surjective, $G$-equivariant and $\phi^{-1}({\bf 0})={\bf 0}$. Note that $\phi$ may not be injective. Now $\F(x_1,x_2,\dots,x_n)^G$ is finitely generated over $\F$ by Theorem~\ref{theo:RationalInvariants}, and $\phi(\F(x_1,x_2,\dots,x_n)^G)=L^G$ by Lemma~\ref{lem:surjectiveInvariants}, so
$L^G$ is generated by $\phi(x_1),\phi(x_2),\dots,\phi(x_n)$ over $\F$.

\end{proof}

\section{Separating invariants}\label{sec:Separating}
Suppose that $G$ is a subgroup of $S_n$ and consider the action of $G$ on $\T(x_1,x_2,\dots,x_n)$ by $\sigma(x_i)=x_{\sigma(i)}$ for all $i$ and $\sigma\in G$.
 We define a (left) action of $G$ on $\R^n$ by
$$
\sigma (v_1,v_2,\dots,v_n)=(v_{\sigma^{-1}(1)},v_{\sigma^{-1}(2)},\dots,v_{\sigma^{-1}(n)}).
$$
For $\alpha\in \N^n\subseteq\R^n$ we have
 $\sigma\cdot x^\alpha=\prod_{i=1} x_{\sigma(i)}^{\alpha_i}=\prod_{i=1}^n x_i^{\alpha_{\sigma^{-1}(i)}}=x^{\sigma(\alpha)}$.
 We can view elements of $\T(x_1,x_2,\dots,x_n)$ as piecewise linear functions on $\R^n$. 
For $v\in \R^n$ we have $x^{\sigma(\alpha)}(v)=\prod_{i=1}^n v^{\alpha_i}_{\sigma(i)}=x^{\alpha}(\sigma^{-1}(v))$.
It follows that 
  $(\sigma\cdot f)(v)=f(\sigma^{-1}(v))$ for all $f\in \T(x_1,x_2,\dots,x_n)$, $v\in \R^n$ and $\sigma\in S_n$.

\begin{definition}
We say that  $f_1,f_2,\dots,f_m\in \T(x_1,x_2,\dots,x_n)^G$
are separating if for all $v,w\in \R^n$ we have:
$f_i(v)=f_i(w)$ for all $i$ if and only if $G\cdot v=G\cdot w$.
\end{definition}
Let $\rho=(n-1,n-2,\dots,0)$. For $\sigma\in S_n$ define $f_\sigma\in \T[x_1,x_2,\dots,x_n]^G$ by $f_\sigma=\Tr_G(x^{\sigma(\rho)})=\sum_{\tau\in G} x^{\tau\sigma(\rho)}$. 
Note that $f_\sigma=f_\lambda$ if and  only if 
$G\sigma=G\lambda$.

\begin{theorem}\label{theo:SeparatingInvariants}
 A set of separating invariants is obtained by taking $e_1,e_2,\dots,e_n$ together with all $f_{\sigma}$, $\sigma\in S_n$.
\end{theorem}
\begin{proof}
Suppose that $v=(v_1,v_2,\dots,v_n),w=(w_1,w_2,\dots,w_n)\in \R^n$ satisfy $e_i(v)=e_i(w)$ for all $i$
and $f_{\sigma}(v)=f_\sigma(w)$ for all $\sigma\in S_n$. We will show that $G\cdot v=G\cdot w$.

Choose a permutation $\gamma\in S_n$ such that
$v_{\gamma(1)}\geq v_{\gamma(2)}\geq \cdots \geq v_{\gamma(n)}$. 
We have $e_k(v)=v_{\gamma(1)}+v_{\gamma(2)}+\cdots +v_{\gamma(j)}$ for all $j$. In particular, we have $e_1(v)=v_{\gamma(1)}$ and $e_j(v)-e_{j-1}(v)=v_{\gamma(j)}$ for $j=2,3,\dots,n$.
Also, for any permutation $\sigma\in S_n$
there is an inequality
\begin{multline*}
 x^{\gamma(\rho)}(v)=(n-1)v_{\gamma(1)}+(n-2)v_{\gamma(2)}+\cdots+v_{\gamma(n-1)}\geq\\
(n-1)v_{\sigma\gamma(1)}+(n-2)v_{\sigma\gamma(2)}+\cdots+v_{\sigma\gamma(n-1)}=x^{\sigma\gamma(\rho)}(v).
\end{multline*}
It follows that $f_\gamma(v)=\Tr_G(x^{\gamma(\rho)})(v)=\max_{\sigma\in G} x^{\sigma\gamma(\rho)}(v)=x^{\gamma(\rho)}(v)$.
Since $e_j(v)=e_j(w)$ for all $j$, $w_1,w_2,\dots,w_n$ is a permutation of $v_1,v_2,\dots,v_n$. 
We have 
\begin{multline}\label{eq:tau}
f_\gamma(v)=x^{\gamma(\rho)}(v)=(n-1)v_{\gamma(1)}+(n-2)v_{\gamma(2)}+\cdots+v_{\gamma(n-1)}\geq\\
(n-1)w_{\tau\gamma(1)}+(n-2)w_{\tau\gamma(2)}+\cdots+w_{\tau\gamma(n-1)}=x^{\tau\gamma(\rho)}(w).
\end{multline}
for all $\tau\in G$. By assumption, $f_\gamma(v)=f_\gamma(w)=
\max_{\tau\in G} x^{\tau\gamma(\rho)}(w)$. So (\ref{eq:tau}) is an equality for some $\tau\in G$.
We get $v_{\gamma(i)}=w_{\tau\gamma(i)}$ for all $i$. 
It follows that $v_i=w_{\tau(i)}$ for all $i$, and $v=\tau^{-1}(w)$. We conclude that $G\cdot v=G\cdot w$.
\end{proof}
\begin{proof}[Proof of Theorem~\ref{theo:3}]
The separating invariants found in Theorem~\ref{theo:SeparatingInvariants}
are of the form $\Tr(x^\alpha)=\sum_{\sigma\in G} x^{\sigma(\alpha)}$. As a function $\R^n\to \R$, $\Tr(x^\alpha)$ is equal to
$$
(v_1,v_2,\dots,v_n)\mapsto \max_{\sigma\in G} \langle v,\sigma(\alpha)\rangle.
$$
So the separating invariants form a  max-filter bank. 
It follows from \cite[Corollary 1.5]{balan2023g} that these separating invariants induce a  bi-Lipschitz embedding of the orbit space into Euclidean space. Theorem~\ref{theo:SeparatingInvariants} gives $n+n!/|G|$ separating invariants, namely, $e_1,e_2,\dots,e_n$ and for every right coset $G\sigma$ in $S_n$ we have an invariant $f_\sigma$.
\end{proof}

\section*{Acknowledgements}
The author was partially supported by NSF grant  DMS 2147769 and a Simons Fellowship.

\printbibliography
\end{document}